\documentclass[12pt]{amsart}

\usepackage{amsmath}
\usepackage{amssymb}
\usepackage{amsthm}
\usepackage{enumerate}
\usepackage{amsbsy}
\usepackage{amsfonts}
\usepackage{graphicx}
\usepackage{color}

\newtheorem{thm}{Theorem}[section]
\newtheorem{prop}[thm]{Proposition}
\newtheorem{lem}[thm]{Lemma}

\theoremstyle{remark}
\newtheorem{rk}{Remark}

\numberwithin{equation}{section}

\newcommand{\al}{\alpha}

\renewcommand{\bar}{\overline}

\renewcommand{\r}{\rho}

\newcommand{\inp}[2]{\langle #1,#2 \rangle}

\newcommand{\Del}[1]{}

\newcommand{\s}{\sigma}

\newcommand{\na}{\nabla}

\newcommand{\supp}{\operatorname{supp}}

\newcommand{\De}{\Delta}

\newcommand{\dist}{\text{ dist }}

\newcommand \mix[2] {L^{#1}_{t}L^{#2}_{x}}
\newcommand \nmix[2] {\|_{L^{#1}_{t}L^{#2}_{x}}}
\newcommand \rnmix[2] {\|_{L^{#1}_{t}\mathfrak L^{#2}_{r}}}
\newcommand \hmix[1] {\|_{L_r^{2}H_\s^{#1}}}
\newcommand \srmix[2] {\|_{L_t^{#1}\mathfrak L^{#2}_rL_\s^2}}
\newcommand \rn {\|_{L^{q}_{t}\mathfrak L^{p}_{r}}}
\newcommand \prn {\|_{L^{\widetilde q'}_{t}\mathfrak L^{\widetilde p'}_{r}}}
\newcommand \pprn {\|_{L^{q'}_{t}\mathfrak L^{ p'}_{r}}}

\newcommand \nqp[2] {\|_{L^{#1}_{t} \mathfrak L^{#2}_{r}}}
\newcommand \wtq {\widetilde q'}

\newcommand \w{\widetilde}
\newcommand \m{\mathcal}

\newcommand \su{ e^{-it \omega(|\nabla|)}  }

\begin{document}
\title[Strichartz estimates in spherical coordinates]{Strichartz estimates  in spherical coordinates}
\author{Yonggeun Cho}
\address{Department of Mathematics, and Institute of Pure and Applied Mathematics, Chonbuk National University, Jeonju 561-756, Republic of Korea}
\email{changocho@jbnu.ac.kr}
\author{Sanghyuk Lee}
\address{Department of Mathematical Sciences, Seoul National University, Seoul 151-747, Republic of Korea}
\email{shklee@snu.ac.kr}
\subjclass[2000]{42B37, 35Q40}
 \keywords{Dispersive equations,
Strichartz estimates, angular regularity.}
\begin{abstract}
In this paper we study Strichartz estimates for dispersive equations
which are  defined by radially symmetric pseudo-differential
operators, and of which  initial data belongs to the spaces of
Sobolev type defined in spherical coordinates. We obtain the space
time estimates on the best possible range including the endpoint
cases.
\end{abstract}

\maketitle

\section{Introduction}
In this paper we consider the Cauchy problem of linear dispersive
equations:
\begin{align}\label{linear eqn}
iu_t - \omega(|\nabla|) u = 0\;\;\mbox{in}\;\;\mathbb{R}^{1+n},\quad
u(0) = \varphi\;\;\mbox{in}\;\; \mathbb{R}^n, \ n \ge 2
\end{align}
where
$\omega(|\nabla|)$ is the pseudo-differential operator of which
multiplier is $\omega(|\xi|)$.
 Typical examples of $\omega$ are $\rho^a\, (0 < a \neq 1)$,
$\sqrt{1 + \rho^2}$, $\rho\sqrt{1 + \rho^2}$, and
$\frac{\rho}{\sqrt{1 + \rho^2}}$ which describe Schr\"{o}dinger type
equation \cite{las}, Klein-Gordon or semirelativistic equation
\cite{frle}, iBq, and imBq (see \cite{chooz0} and references
therein).

The solution can formally be given by
$$u(t,x) = \frac1{(2\pi)^n} \int_{\mathbb{R}^n} e^{i(x\cdot \xi - t \omega(|\xi|))} \widehat{\varphi}(\xi)\,d\xi.$$
Here $\widehat{\varphi}$ is the Fourier transform of $\varphi$
defined by $\int_{\mathbb{R}^n }e^{-ix\cdot \xi} \varphi(x)\,dx$.
There has been a lot of work on the space time estimates for the
solution  $u$ which play important roles in recent studies on
nonlinear dispersive equations. (See  Cazenave \cite{caz}, Sogge
\cite{sog} and Tao \cite{tb} and references therein.) Especially,
when $\omega(\rho)=\rho^a$, $a\neq 0$, the solution satisfies
\begin{equation}\label{basic}
 \|u\|_{\mix{q}{p}}\le C\|\varphi\|_{\dot{H}^s}
\end{equation}  for some $p,$ $q$ with $s=\frac
n2-\frac{n}p-\frac aq$, which is known as Strichartz estimates.
These estimates were first established by Strichartz \cite{str} for
$q =p$ and were generalized to mixed norm $(q\neq p)$ spaces by
Ginibre and Velo \cite{gv1, gv2} except the endpoint cases, which
were later proven by Keel and Tao \cite{kt}.

It is well known that the estimate \eqref{basic} is possible only if
$n/p+2/q\le n/2$, $q\ge 2$ when $a > 0$, $a\neq 1$, and
$\frac{n-1}p+\frac 2q\le \frac{n-1}2$, $q\ge 2$ when $a=1$ as it can
be easily seen by Knapp's examples. In  applications of
\eqref{basic}, depending on various problems being considered, the
existence of proper $(p,q)$ for which \eqref{basic} holds  is
important. Hence, there have been attempts to extend the range $p,q$
via a suitable generalization \cite{tao, ster} while the natural
scaling structure remains unchanged.  As it was observed in
\cite{ster, sh}, the estimates have wider ranges of admissible $p,q$
when $\varphi$ is a radial function. This seems to make sense in
that Knapp's examples are non-radial. However, to make these
estimates on the extended range hold for general functions which are
no longer radial, such extension should be compensated with extra
regularity in angular direction.

For precise description we now define
a function spaces of Sobolev type in spherical
coordinates. Let $\De_\s=\sum_{1\le i<j\le n} \Omega_{i,j}^2$, $\Omega_{i,j}=x_i\partial_j-x_j\partial_i$,  be the Laplace-Beltrami operator defined
on the unit sphere in $\mathbb R^n$ and set  $D_\s = \sqrt{1 - \Delta_\s}$.
For $|s| < n/2,\;\; \al \in \mathbb{R},$ we denote by $\dot{H}_r^{s}H_\s^{\al}$ the space
\[ \Big\{f \in \mathcal S' :
\|f\|_{\dot{H}_r^{s}H_\s^{\al}} \equiv \|\,|\na|^sD_\s^\al
f\|_{L^2}<\infty \Big\}.
\]
(It should be noted that $C_c^\infty$ is dense in
$\dot{H}_r^{s}H_\s^{\al}$ since $|s| < n/2$.)
A natural generalization of \eqref{basic} is the estimate
\begin{equation}\label{angle}
 \|u\|_{\mix{q}{p}}\le C\|\varphi\|_{\dot{H}_r^{s}H_\s^{\al}}.
\end{equation}
In fact, for the wave equation $(\omega(\rho)=\rho)$ Strebenz
\cite{ster} obtained \eqref{angle} on almost optimal range of $q,r$
up to the sharp regularity (see also Section \ref{wavet}). In
\cite{jiwayu} \eqref{angle} was shown for $\frac1q <
(n-1)(\frac12-\frac1p), q \ge 2$, and $\al \ge \frac1q$,  when
$\omega(\rho) = \rho^a$  by utilizing Rodnianski's argument in
\cite{ster} and weighted Strichartz estimates (see \cite{chozsash,
fawa, choleeoz}). Recently, Guo and Wang \cite{guwa} considered the
estimate \eqref{angle} with $\omega(\rho)=\rho^a$ and  radially
symmetric initial data, and obtained \eqref{angle} on the  optimal
range of $p,q$ except some endpoint cases.

In a different direction one may try to extend  \eqref{basic} to include  more general $\omega$.  Let us consider
$\omega\in C^\infty(0,\infty)$ which
satisfies the following properties:
\begin{align}
\tag{$i$}\label{con1}
 &\omega'(\rho) > 0, \text{ and either } \omega''(\rho) > 0 \text{ or } \omega''(\rho) < 0 ,
 \\
\tag{$ii$}\label{con2} &|\omega^{(k)}(\rho_1) |\sim
|\omega^{(k)}(\rho_2)| \;\ \text{ for } 
\;\; 0 < \rho_1 < \rho_2 < 2\rho_1,
\\
\tag{$iii$}\label{con3} &\rho|\omega^{(k+1)}(\rho)| \lesssim
|\omega^{(k)}(\rho)|. 
\end{align}

 We also define
 a pseudo-differential operator $\mathcal D_\omega^{s_1, s_2}$ by setting
  $$\mathcal F({\mathcal D_\omega^{s_1, s_2}} f)(\xi)=\left(\frac{\omega'(|\xi|)}{|\xi|}\right)^{s_1} |\omega''(|\xi|)|^{s_2} \widehat f(\xi).$$
Here $\mathcal F$ denotes the Fourier transform. In \cite{choozxia}
(also see \cite{GNT} for earlier result), the authors proved the
following: If $\omega$ satisfies the conditions \eqref{con1},
\eqref{con2} for $k=1,2,$ and \eqref{con3} for $1\le k\le [\frac
n2]+1$\footnote{In \cite{choozxia} the condition \eqref{con3} was
assumed for $k\ge 1$ to get \eqref{homoo} but \eqref{con3} for $1\le
k\le [\frac n2]+1$  is enough as it is clear from the proof in
\cite{choozxia}.}, then
\begin{align}
\label{homoo} \| u \nmix{q}{p}\lesssim \|\mathcal D_\omega^{s_1, s_2}
\varphi\|_{\dot{H}^s}
\end{align}
holds for $2 \le p,q\le\infty$, $\frac2q+\frac np \le \frac n2$ and
$(n, p, q)\neq (2, \infty, 2)$ with
\begin{equation}
\label{scon}
s_1 =  (\frac14 - \frac1{2p})-\frac1q, \ s_2
=\frac1{2p}-\frac14, \ s = n(\frac 12 - \frac 1p)- \frac2q.
\end{equation}
As it can be seen by considering homogeneity the range and the
exponents are clearly sharp (\cite{choozxia}).

In this note we  intend to study the estimates \eqref{basic} and
\eqref{angle} within a unified  framework. More precisely,  we
consider the estimate
\begin{align}
\label{homo} \| u \nmix{q}{p}\lesssim \|\mathcal D_\omega^{s_1, s_2}
\varphi\|_{\dot{H}_r^sH_\s^\al}
\end{align}
which has a wider range $p,q$ of boundedness than \eqref{homoo}. By
allowing some regularity loss in the spherical variables, we want to
find the optimal range of $p,q$ for which \eqref{homo} holds. In
fact, using a Knapp type example which is adapted to radial function
one can see that \eqref{homo} is possible  only if
\begin{equation}\label{nec}
\frac1q
\le \frac{2n-1}{2}(\frac12 - \frac1p).
\end{equation}
(See Section \ref{neccon}.) Since we already have the usual
Strchartz estimates \eqref{homoo} on the range $ \frac2q+\frac np\le
\frac n2$, we are mainly interested in the estimates for $(p,q)$
which are contained in the region $ \frac2q+\frac np>\frac n2$.

\medskip

  The following is
our first result which establishes \eqref{homo} in the best possible range of $p,q$ except an endpoint.
\begin{thm}\label{opt}  Let  $n \ge 2, 2 \le p, q \le \infty$ and $s_1, s_2, s$ given by  \eqref{scon}.   Suppose that $\omega\in C^\infty(0,\infty)$ satisfies the conditions
\eqref{con1},  \eqref{con2} for $k= 1, 2$, and \eqref{con3} for
$1\le k \le \max(4, [\frac n2]+1)$. If $\frac n2(\frac12-\frac1p)
\le \frac1q \le \frac{2n-1}{2}(\frac12- \frac1p)$,   $(n, p, q) \neq
(2, \infty, 2)$,  and $(p, q) \neq (\frac{2(2n-1)}{2n-3}, 2)$, then
 the solution $u$ to
\eqref{linear eqn} satisfies \eqref{homo}
for $\alpha > \frac{5n-1}{5n-5} (\frac np+\frac 2q-\frac n2 ) $.
\end{thm}

The aforementioned $\rho^a\, (0 < a \neq 1)$, $\sqrt{1 + \rho^2}$,
$\rho\sqrt{1 + \rho^2}$, and $\frac{\rho}{\sqrt{1 + \rho^2}}$
satisfy the conditions \eqref{con1}, \eqref{con2}, and \eqref{con3}.
Hence Theorem \ref{opt} gives various new  estimates for
Schr\"{o}dinger type equation, Klein-Gordon or semirelativistic
equation, iBq, and imBq. Theorem \ref{opt} generalizes Shao's
results in \cite{sh} where $\omega(\rho) = \rho^2$ and radial data
were considered. In \cite{guwa}, the estimates for $(p,q)$ on the
sharp line $\frac1q = \frac{2n-1}{2}(\frac12- \frac1p)$ were
obtained when $p\le q$, $\omega(\rho)=\rho^a$ and  the initial datum
$\varphi$ is radial. Our result include all the  estimates on the
sharp line except for $(p, q) \neq (\frac{2(2n-1)}{2n-3}, 2)$, which
is left open and seems to be beyond the method of this paper.
  It should be noted that if
$\omega$ satisfies \eqref{simsim} below instead of \eqref{con3}, 
then $\mathcal D_\omega^{s_1, s_2}$ simplifies so that
$\mathcal D_\omega^{s_1, s_2} \sim |\omega''(|\nabla|)|^{-\frac1q}$.
Although Theorem \ref{opt} gives a sharp estimate in $(q, p)$ pairs,
there is no reason to believe that the angular regularity is sharp.
Substantial improvement will be possible by obtaining more refined
estimates for Bessel function.

When $\omega$ satisfies
\begin{equation}\tag{$iv$}\label{simsim}\rho|\omega''(\rho)| \sim
|\omega'(\rho)|\;\ \text{for}\; \rho>0,\end{equation} \eqref{homo}
can be shown with the less angular regularity. (See Section
\ref{improve}). The following is our second result which improves
the  result in \cite{jiwayu}.

\begin{thm}\label{opt2}
Let  $n \ge 2, 2 \le p, q \le \infty$ and $s_1, s_2, s$ given by
\eqref{scon}.   Suppose that $\omega\in C^\infty(0,\infty)$
satisfies the conditions \eqref{con1}, \eqref{con2} for $k = 1, 2$, \eqref{con3} for $1\le k\le \max(4, [\frac n2]+1)$, and \eqref{simsim}.
If $\frac n2(\frac12-\frac1p) \le \frac1q \le
\frac{2n-1}{2}(\frac12- \frac1p)$,   $(n,p, q) \neq (2,\infty, 2)$, and
$(p, q) \neq (\frac{2(2n-1)}{2n-3}, 2)$,
 the solution $u$ to
\eqref{linear eqn} satisfies \eqref{homo}
for $\alpha > \frac{1}{2}\frac{2n-1}{n-1} (\frac np+\frac 2q-\frac n2 ) $.
\end{thm}

 Compared to the previous works,  our approach here  is simpler and  more systematic.
Especially, it enables us to provide a short  proof of the result in
\cite{ster} (see Section \ref{wavet}). By spherical harmonic
expansion the matters basically reduce to one dimensional situation
but it involves with a family of operators which are given by Bessel
functions of different orders. To get the desired estimate, the
growth of bounds depending on the orders needs to be effectively
controlled in a uniform way. It will be done by comparing spatial
scale and the orders of Bessel functions. Our novelty here is the
use of a temporal localization (see Lemma \ref{local t}) which is
available only after frequency and spatial localizations. This kind
of localization was first observed  by the second author \cite{lee}
and similar idea was used to study the space time estimates for the
Schr\"odinger equation \cite{lcv, lr}. This make it possible to
reduce the time local estimate to that of the same scale in spatial
space so that it suffices to work on the estimate which is local in
both time and spatial spaces, and this localization also plays a
role in obtaining precise estimates for general $\omega$. (See
Section \ref{localsf}.)

This paper is organized as follows: In Section 2 we consider the
asymptotic behavior of Bessel function. In Section 3 we obtain
various preliminary estimates via space-frequency-time localization
which are expected to be useful for related problems and in Section
4 the proofs of Theorem \ref{opt}, \ref{opt2} are given.

Throughout the paper, if not specified, $A \lesssim B$, $A \sim B$
mean $A \le CB$, $C^{-1}A \le B \le CB$, respectively,  for some
generic constant $C$.

\section{Estimates for Bessel functions}

For the proofs of theorems we need estimates for Bessel functions
$J_\nu$, which depend on $\nu$. When $\nu$ is bounded,
estimates are  easy to obtain.  We start by recalling
some basic properties of Bessel functions.

     Let  $\nu_0 > 1$ be a fixed number.
 If $0 \le \nu \le \nu_0$, then
\begin{align}
&\qquad\qquad\quad |J_\nu(r)| \lesssim 1, \;\;\mbox{if}\;\;r \lesssim 1,\label{lower}\\
&J_\nu(r) = r^{-\frac12}(b_+e^{ir} + b_-e^{-ir}) +
\Psi(r),\;\;\mbox{if}\;\;r \gg 1, \label{upper}
\end{align}
where $|\Psi(r)| \lesssim r^{-\frac32}$. For instance see \cite{st,
Wat}. If $\nu > \nu_0$, then we have
\begin{align}
|J_\nu(r)| \lesssim \exp(-C\nu),\;\;\mbox{if}\;\;r \ll
\nu.\label{bessel1}
\end{align}
This is easy to show by making use of the Poisson representation
(\cite{st, Wat})
\begin{align*}
J_\nu(r) = \frac{ (\frac{r}{2})^\nu }{ \Gamma(\nu + \frac{1}{2} )
\Gamma( \frac{1}{2} ) } \int_{-1}^{1} e^{irs}(1 - s^2)^{\nu -
\frac{1}{2} } ds
\end{align*}
and Stirling's formula (\cite{leb}) $\Gamma(t) \sim
\sqrt{2\pi}t^{t-\frac12}e^{-t}$ for large $t$.

For simplicity we  denote $z+\bar z$ by
\[z+\mathcal C.\mathcal C\]
so that $\mathcal C.\mathcal C$ stands for the complex conjugate of terms
appearing before $+ \mathcal C.\mathcal C$. We now make use of the following
representation of Bessel function (see \cite{barcor} or Lemma 3 of
\cite{baruve});
\begin{align}\label{bessel2nu1}
J_\nu(r) = 2(r^2-\nu^2)^{-1/4}\left(c_\nu e^{i\theta(r)}+\mathcal
C.\mathcal C\right) + h_\nu(r),
\end{align}
where $$\theta(r) = r\left[\left(1 -
\frac{\nu^2}{r^2}\right)^{\frac12} - \frac\nu{r}\left(\frac{\pi}2 -
\cos^{-1}\frac \nu{r}\right)\right],\quad |h_\nu(r)| \lesssim
r^{-1}.$$ It can be obtained  by the stationary phase method and
Schl\"{a}fli's integral representation (see p.176, \cite{Wat})
which is given by
\begin{align}\label{schlafli}
J_\nu(r) = \frac1{\pi} \int_0^{\pi} e^{i(r\sin \theta - \nu \theta)}
\,d\theta - \frac{\sin (\nu \pi)}{\pi} \int_0^\infty e^{-\nu \tau -
r \sinh \tau}\,d\tau.
\end{align}
The following lemma gives  asymptotic bounds for  $J_\nu$ when
$\nu$ is large.

\begin{lem}\label{asymp bessel}
Let $\nu \ge \frac12$. Then the following holds:
\begin{align}
&\qquad\quad \qquad\qquad\quad |J_\nu(r)| \lesssim r^{-\frac12}, \;\;\mbox{if}\;\;r \ge 2\nu\label{bessel3},\\
& J_\nu(r) = (c_\nu r^{-1/2} + \widetilde{c_\nu}\, \nu^2\,r^{-\frac32}) e^{ir} +
\mathcal C.\mathcal C+ \Psi_\nu(r),\;\;\mbox{if}\;\;r \ge 4\nu^\frac{8}{5},\label{bessel4}
\end{align}
where $
c_\nu = \frac{e^{-i(\frac\pi4 + \frac{\nu\pi}2)}}{2\sqrt{2\pi}},$ $  \widetilde{c_\nu} = -2ic_\nu,$ and $|\Psi_\nu(r)| \lesssim r^{-1}.$
\end{lem}
\begin{proof}[Proof of Lemma \ref{asymp bessel}]
We rewrite $e^{i\theta(r)}$ as
$$
e^{i\theta(r)} = e^{ir}\Big(1 + i(\theta(r) - r)\Big) +
e^{ir}\Big(e^{i(\theta(r) - r)} - 1 - i(\theta(r) - r)\Big).
$$
Substituting this into \eqref{bessel2nu1}, we obtain
\begin{align*}
J_\nu(r) =
r^{-\frac12}\left(1-\frac{\nu^2}{r^2}\right)^{-1/4}\Big(c_\nu
e^{ir}(1 + i(\theta(r) - r)) + \mathcal C.\mathcal C\Big) +
\widetilde{h_\nu}(r),
\end{align*}
where $$\widetilde{h_\nu}(r) = h_\nu(r) +
\left(c_\nu\frac{e^{ir}(e^{i(\theta(r)-r)}) - 1 -
i(\theta(r)-r)}{(r^2-\nu^2)^\frac14} + \mathcal C.\mathcal
C\right).$$ Let $\theta_1(r) = \big(1-\frac{\nu^2}{r^2}\big)^{-1/4}
- 1$ and $\theta_2(r) = i(\frac{3\nu^2}{2r} + \theta(r) - r)$. Then
\begin{align*}
J_\nu(r) &= r^{-\frac12}(1 + \theta_1(r))\left(c_\nu e^{ir}\big(1 - i\frac{2\nu^2}{r} + \theta_2(r)\big) + \mathcal C.\mathcal C\right) + \widetilde{h_\nu}(r)\\
&= r^{-\frac12}\left(c_\nu e^{ir}\Big(1 - i\frac{2\nu^2}{r}\Big)
+ \mathcal C.\mathcal C\right) + \Psi_\nu(r),
\end{align*}
where
\begin{align*}
\Psi_\nu(r) = \widetilde{h_\nu}(r) &+ r^{-1/2}\theta_1(r)
\left(c_\nu e^{ir}\Big(1 - i\frac{3\nu^2}{2r} + \theta_2(r)\Big) + \mathcal C.\mathcal C\right)\\
&+ r^{-\frac12}(1 + \theta_1(r))\left(c_\nu e^{ir}\theta_2(r) +
\mathcal C.\mathcal C\right).
\end{align*}
Taylor's theorem gives that $|\widetilde{h_\nu}(r)| \lesssim
r^{-1}$, $|\theta_1(r)| \lesssim \frac{\nu^2}{r^2}$ and
$|\theta_2(r)| \lesssim \frac{\nu^4}{r^3}$. Hence
$|\Psi_\nu(r)| \lesssim r^{-1}$ for $r \ge 4\nu^\frac{8}{5}$.
This
completes the proof of Lemma \ref{asymp bessel}.
\end{proof}

\section{Estimates via space-frequency localization}\label{localsf}

In this section we obtain estimates via localization on both space
and frequency sides. Let  $0 < \lambda_0 \lesssim 1$ and
 let $\varpi \in C^4(1/2,2)$ satisfy that
\begin{align}\begin{aligned}\label{cond ome 1st}
|\varpi'(\r)| \sim 1,\;\; \lambda_0 \le |\varpi''(\rho)| \lesssim 1,
\;\;|\varpi^{(3)}(\rho)| + |\varpi^{(4)}(\rho)| \lesssim 1
\end{aligned}\end{align}
if $1/2 < \rho < 2$. For $R > 0$, let us set
$\chi_{R}=\chi_{\{x:R\le |x|< 2R\}}$ and  define
\begin{align}\label{th-dfn}
\mathcal T_R^\nu h(t, r) = \chi_R(r)r^{-\frac{n-2}2}\int
e^{-it\varpi(\rho)} J_\nu(r\rho)\beta(\rho)h(\rho) d\rho,
\end{align}
where $\beta\in C_c^\infty(1/2, 2)$. In what follows $\beta$ may be
different  at each occurrence  but we keep the same notation as long
as it is contained  uniformly in  $ C_c^\infty(1/2, 2)$.

We denote  by $\mathfrak L_{r}^p$  the space $L^p(r^{n-1}dr)$.

\begin{prop} \label{linear} Let $R \gtrsim 1$ and $\mathcal T_R^\nu$ be defined by \eqref{th-dfn}.  If
$2\le p,q\le \infty$, $\nu \ge 0$, and $2/q\ge 1/2-1/p$, then there
is a constant $C = C(n, p, q) > 0$, independent of $\lambda_0, \nu,$
$R$, such that
\begin{equation}\label{single}\|\mathcal T_R^\nu h\rn\le
C
\lambda_0^{-\frac12(\frac{1}{2}-\frac1p)}(1+\nu)^{\frac{4}{5}(\frac12-\frac1p)}
R^{\frac1q-\frac{2n-1}{2}(\frac12-\frac1p)}\|h\|_2.
\end{equation}
\end{prop}

The exponent on $R$ is sharp as it can be shown by using the example
for the necessary condition \eqref{nec} (see Section \ref{neccon}).

For the proof, we need to show the cases $(p,q)=(2,\infty)$,
$(2,2)$, $(\infty,2)$, and $(\infty, 4)$ since the other estimates
follow from interpolation. The case $(p,q)=(2,\infty)$ is
straightforward form Parseval's theorem for Hankel transforms (for
example, see Theorem 3 in \cite{Of}). Then  the case $(p,q)=(2,2)$
follows from the square function estimates for Bessel functions such
that
\begin{align}\label{L2-bessel}
\int_{r \sim R} |J_\nu(r)|^2\,dr \lesssim 1
\end{align}
for all $\nu \ge 0$ and $R > 0$.  This can be shown by using
\eqref{lower}, \eqref{upper} when $0 \le \nu \le \nu_0$,
$\int_0^\infty |J_\nu(r)|^2dr/r = 1/(2\nu)$ when $\nu \ge \nu_0$ and
$R \lesssim \nu$ (see p. 405 of \cite{Wat}), and \eqref{bessel3}
when $R \gg \nu > \nu_0$.

Now, by Schwarz's inequality and \eqref{L2-bessel} we have
\begin{align}\label{lrinfty}
|\mathcal T_R^\nu h (t, r)| \lesssim R^{-\frac{n-2}2} \sup_{r\sim
R}\left(\int_{\r\sim 1}|J_\nu(r\r)|^2\,d\r\right)^\frac12\|h\|_2
\lesssim R^{-\frac {n-1}2}\|h\|_2.
\end{align}
Similarly,
since $|\varpi'(\rho)| \sim 1$, by the change of variables
$\varpi(\r) \mapsto \r$, Plancherel's theorem in $t$ and Schwarz's
inequality we have
\begin{align}\label{lr2}
\|\mathcal T_R^\nu h\|_{L_t^2 {\mathfrak L^2_r}}^2 &\lesssim \int_{\rho \sim 1}
|h(\varpi^{-1}(\rho))|^2\int_{r \sim R}|J_\nu(r\rho)|^2r\,dr\,d\rho.
\end{align}
By \eqref{L2-bessel} and reversing the change of variables we obtain
\eqref{single} for $(p,q)=(2,2)$.

It remains to show \eqref{single} for $(\infty,2)$, and $(\infty,
4)$. For this purpose we use a localization property of $\mathcal
T_R^\nu$ in $t$,  which is possible because $\supp \beta\subset
(1/2,2)$.

\subsection{Temporal localization}

Let $\widetilde \chi_R$ be a measurable function supported $[0, 2R]$
with $\|\widetilde \chi_R\|_\infty\le 1$.  Let $\widetilde \omega
\in C^4(1/2,2)$ and let us define
\begin{align*}
\mathcal R_R^\nu\, g(t,r)= \widetilde \chi_R(r)\int
e^{-it\widetilde{\omega}(\rho)} J_\nu(r\r) \beta(\rho) g(\rho)
d\rho.
\end{align*}
Making use of the localization of scale $R$ in $r$  and the fact that
$\supp \beta\subset (1/2,2)$, it is possible to localize  the estimate for
$\mathcal R_R^\nu$ in $t$ at the same scale $R$. It is crucial for obtaining
sharp estimates for $\mathcal T_R^\nu$.

\begin{lem}\label{local t}  Let $R \gtrsim 1$, $\nu=\nu(k) = \frac{n-2+2k}{2},$ $k=0,1,2,\dots, $ and let $I$ be an interval of length $R$.
Suppose that $|\widetilde \omega'(\r)| \sim 1,\;\; |\widetilde
\omega^{(k)}(\rho)| \lesssim 1,\; k = 2,3,4,$ on the support of
$\beta$.  And suppose that
\begin{equation*}
\|\chi_I(t)\mathcal R_R^\nu\, g\rnmix{q}{p}\le B R^b\,\|g\|_{2}
\end{equation*}
for  $2\le p, q\le \infty$, $B\gtrsim 1$, and $b  \ge  \frac
np+\frac1q-1$. Then $\|\mathcal R_R^\nu\, g\rnmix{q}{p}\le CB
R^b\,\|g\|_{2}.$
\end{lem}

\noindent Note that $\mathcal T_R^\nu
h(r,t)=r^{-\frac{n-2}2}\mathcal R_R^\nu h(r,t)$ with a proper
$\widetilde \chi_R$ if $\varpi=\w\omega$. From Section \ref{neccon}
we see that a lower bound for  $\mathcal T_R^\nu: L^2\to L^q_tL^p_x$
is $CR^{\frac1q+\frac{2n-1}{2p}-\frac34}$ which is bigger than or
equal to $CR^{\frac np+\frac1q-1}$  if $p\ge 2$ and $R$ is large
enough. So we can use this localization for $\mathcal T_R^\nu$
without any loss in bound.

\begin{proof}[Proof of Lemma \ref{local t}]
Note that $\nu = 0$ or $\nu \ge 1/2$ from the definition of
$\nu(k)$. By Plancherel's theorem it suffices to show that
\begin{equation}\label{modi}
\|\chi_I(t)\mathcal R_R^\nu\, \widehat g\rnmix{q}{p}\le B
R^b\,\|g\|_{2}
\end{equation}
for $B\gtrsim 1$, $b  \ge  \frac np+\frac1q-1$ implies \[\|\mathcal
R_R^\nu\, \widehat g\rnmix{q}{p}\le CB  R^b\,\|g\|_{2}.\]

Using the integral representation \eqref{schlafli} of Bessel
function  we write
\begin{equation}
\label{repre}
 \mathcal R_R^\nu\, \widehat g = \mathcal R_1 g + \mathcal R_2 g,
\end{equation}
where
\begin{align*}
\mathcal R_1 g&= \frac{\widetilde\chi_R(r)}{\pi} \int_0^{\pi} \int
e^{-it\widetilde\omega(\r)}  e^{i(r\r\sin \theta - \nu
\theta)} \,\beta(\r)\widehat g(\rho)\, d\rho\,d\theta,\\
\mathcal R_2 g&=- \frac{\sin (\nu \pi)}{\pi} \widetilde\chi_R(r)
\int_0^\infty \int
 e^{-it\widetilde\omega(\r)}   e^{-\nu \tau - r\r \sinh \tau}\,\beta(\r)\widehat g(\rho)\, d\rho\,d\tau.
\end{align*}
Note that if $\nu = 0$, then $\mathcal R_2 = 0$.

 If $\nu \ge 1/2$, then we claim that for  $2\le p, q\le \infty$
\begin{equation}\label{errr} \|\mathcal R_2 g\rnmix{q}{p}\le CR^{\frac np+\frac1q-1}\|g\|_{2}.
\end{equation}
In view of interpolation, it is sufficient to consider the cases
$(p,q)=(2,2), (2,\infty), (\infty,2),$ and $ (\infty,\infty)$. By
H\"older's inequality we only need to show \eqref{errr} for $(p,q)=
(\infty,2), (\infty,\infty).$ Using the fundamental theorem of
calculus and H\"older's inequality, we have
\begin{align*}
\sup_{0\le r\le 2R} |\mathcal R_2 g(t,r)|
&\le |\mathcal R_2 h(t,3R/2)|+\int_{R}^{2R}|\partial_r \mathcal R_2 g(t,r)| dr\\
&\le |\mathcal R_2 h(t,3R/2)|+ CR^\frac12 \|\partial_r \mathcal R_2
g(t,\cdot)\|_{L^2(0, 2R)}.
\end{align*}
Since $|\widetilde \omega'(\r)|\sim 1$, by the change of variables
$\widetilde\omega(\r)\mapsto \r$ and Plancherel's theorem
\begin{align*}\
\|\mathcal R_2 g(t,3R/2)\|_{L^2_t} &\le   C \int_0^\infty
e^{-\nu\tau} e^{-\frac32R\sinh \tau} d\tau \|
 \,g\|_{2} \le \frac{C}{\nu + R}\| \,g\|_{2}.
  \end{align*}
To handle $\|\partial_r \mathcal R_2 g(t,\cdot)\|_{L^2(R,2R)}$,
observe that
\begin{align*}\partial_r \mathcal R_2 g(t,r)=  \frac{\sin (\nu \pi)}{\pi} \widetilde\chi_R(r) \int
 e^{-it\rho} \Big(  \int_0^\infty   e^{-\nu \tau - r\phi(\rho) \sinh \tau} \sinh \tau\, d\tau \Big)\\
 \times  \phi(\rho) \beta(\phi(\r))\widehat g(\phi(\rho))\phi'(\r)\, d\rho,\end{align*}
 where $\phi = \widetilde \omega^{-1}$.
Since $\phi'(\r) \sim 1$, by Plancherel's theorem again  we have
  \begin{align*}
  \|\partial_r \mathcal R_2 g\|_{L_t^2L^2_{dr}(R, 2R)} &\le C \int_0^\infty e^{-\nu\tau} \|e^{-r\sinh \tau}\|_{L^2_{dr}(0, 2R)} \sinh \tau\, d\tau \|
 \, g\|_{2}\\
 &\le C\int_0^\infty e^{-\nu\tau} e^{-R \sinh \tau} (\sinh \tau)^\frac12\, d\tau \|
 \, g\|_{2}\\
 &\le \frac{C}{\nu-1/2 + R}\|\, g\|_{2}.
 \end{align*}
Since $\nu \ge 1/2$, we get the desired estimate \eqref{errr} for
$(p,q)=(2,\infty)$. The estimate for $(p,q)=(\infty,\infty)$ is
straightforward because
\[|\mathcal R_2 g|\le C\int_0^\infty
 e^{-\nu \tau - \frac{R}2 \sinh \tau}d\tau \|g\|_2.\]

By \eqref{repre}, \eqref{errr}, and \eqref{modi} it is now enough to
show that $\|\chi_I(t)\mathcal R_1 g\rnmix{q}{p}\le BR^b\|g\|_{2}$
implies $\|\mathcal R_1 g\rnmix{q}{p}\le CBR^b \|g\|_{2}.$ Since
$|\widetilde \omega'(\r)|\sim 1$, by the change of variables
$\rho\to \phi(\r) = \widetilde \omega^{-1}(\r)$ and Plancherel's
theorem,  we may replace $\mathcal R_1$ with $\w {\mathcal R_1}$
which is given by
\[\w {\mathcal R_1} g= \chi_R(r)\int_0^{\pi} \int e^{-it\rho}  e^{i(r\phi(\r)\sin \theta - \nu
\theta)} \,\widetilde\beta(\r)\widehat g(\rho)\, d\rho\,d\theta\]
for some $\widetilde\beta\in C_c^\infty(1/2,2)$. The matter reduces
to showing that
 \begin{equation}\label{siii}
\|\chi_I(t)\w {\mathcal R_1} g\rnmix{q}{p}\le BR^b\|g\|_{2}
\end{equation}
for $b\ge -1$ implies  \begin{equation} \label{sii} \|\w{\mathcal
R_1} g\rnmix{q}{p}\le CBR^b \|g\|_{2}.
\end{equation}
We note that
\[\w{\mathcal R_1} g(t,r)= \widetilde\chi_R(r) (K_{r}\ast g)(t)\]
 with
\[K_{r}(t)=\frac{1}{2\pi}\int_{0}^{\pi}\!\!\!\int  e^{-i(t\r-r\phi(\r)\sin\theta)} \widetilde\beta(\r)d\rho \,\,e^{-i\nu\theta} d\theta.
\]
  Since $0\le r\le 2R$ and
$|\frac{d}{d\rho}  (t\widetilde \omega(\r)-r\r\sin\theta)|\ge C|t|$
for some $C>0$ if $|t|\ge M R$ for some large $M \gg 1$, from the
condition on $\w\omega$ and  integration by parts (three times) it
follows that for $0\le a\le 3,$
\begin{equation}\label{ker}
| K_r(t)|\le C(1+|t|)^{-3} \le CR^{-a}(1+|t|)^{-3+a}
\end{equation}
if $r\sim R$ and $|t| \ge M R$. Now the argument is rather standard.
Indeed, let $\{I\}$ be a collection of disjoint intervals with
sidelength $\sim R$ which partition $\mathbb R$. Let us denote by
$\widetilde I$ the interval $\{t: \dist(t, I) \le 5MR\}$. Breaking
$g=\chi_{\widetilde I} g+\chi_{\widetilde I^c} g$, by \eqref{ker} we
see that for $t \in I$
\[  |\w{\mathcal R_1} g(t,r)|\le |\w{\mathcal R_1} (\chi_{\widetilde I}\, g)(t,r)| +CR^{-a} \mathcal E\ast |g|(t), \]
where $\mathcal E(t)=(1+|t|)^{-3+a}$. Thus, when $q\neq \infty$,
taking $a=1$, we see that
\begin{align*}
\int_{\mathbb R} \|\mathcal R_1 g(t,\cdot)\|_{\mathfrak L_r^p}^q dt &
\le \sum_{I} \int_{I} \|\mathcal R_1 g(t,\cdot)\|_{\mathfrak L_r^p}^q\,dt\\
&\le C\sum_{I} \int_{I} \|\mathcal R_1  (\chi_{\widetilde I}\, g)\|_{\mathfrak L_r^p}^q\,dt + C R^{-q}\int_{\mathbb R}(\mathcal E\ast | g|)^q(t)\,dt \\
&\le CB^q R^{qb}\sum_{I} \|\chi_{\widetilde I}\, g\|_2^q +
CR^{-q}\|g\|_2^q
\\
&\le CB^q R^{qb}\|g\|_2^q .
\end{align*}
 For the third inequality we use the hypothesis \eqref{siii} and the fact that
 $\mathcal E\in L^1\cap L^\infty$ and for the last inequality we use
 the fact that $b  \ge  \frac np+\frac1q-1$. Hence summation along $I$ gives the
desired estimate \eqref{sii}. When $q = \infty$, the argument is
even simpler. We omit the  detail.
\end{proof}

 Now let us set
\begin{equation}\label{rdef}
\mathcal R_\Omega g(t,r)= \widetilde \chi_R(r)\int
e^{-it\widetilde{\omega}(\rho)} \Omega(r\r) \beta(\rho) g(\rho)
d\rho.
\end{equation}
 From the proof of Lemma \ref{local t} it is obvious that the same statement remains valid even if we replace $R_R^\nu$ by
$\mathcal R_{\Psi_\nu}g$ provided that $R\ge 4\nu^\frac85$. Here
$\Psi_\nu(r)= J_\nu(r) -[ (c_\nu r^{-1/2} + \widetilde{c_\nu}\,
\nu^2\,r^{-\frac32}) e^{ir} + \mathcal C.\mathcal C]$ which is given
in \eqref{bessel4}.  In fact, since we already have \eqref{errr},
one needs  to check that
\begin{align*}
K_{r}(t)= \frac{1}{2\pi}\int  e^{-it\r} \bigg(\int_{0}^{\pi}\!\!\! &e^{i(r\phi(\r)\sin\theta-\nu\theta)} d\theta\\
 &-\bigg[ \Big(\frac{c_\nu} {(\phi(\r) r)^{\frac12} }+ \frac{\widetilde{c_\nu}\nu^2}{(\phi(\r) r)^{\frac32}}\Big) e^{ir\phi(\r)} + \mathcal C.\mathcal C\bigg]\bigg)\beta(\r)d\rho
\end{align*}
satisfies \eqref{ker} if $r\sim R$, $|t| \ge M R$.  It is easy to
see this by making use of the fact that $R\ge 4\nu^\frac85$. One can
handle each term separately. Then the rest of the argument is
straightforward. The similar implication is also valid for $\mathcal
R_\Omega$ with $\Omega= \r^{-\frac12} e^{\pm i\rho}, $
$\nu^2\r^{-\frac32} e^{\pm i\rho}.$ For future use we summarize it
as follows.

\begin{lem}\label{hh}
 Let $R \gtrsim 1$, $\nu$,  $I$ and $\w \omega$  be the same as in Lemma \ref{local t}
 and let $\m R_\Omega$ be defined by \eqref{rdef} with $\Omega=\r^{-\frac12} e^{\pm i\rho}, $
$\nu^2\r^{-\frac32} e^{\pm i\rho}$, $\Psi_\nu$, repectively.
  Suppose that $R\ge 4\nu^\frac85$ and $\|\chi_I(t)\mathcal R_\Omega\, g\rnmix{q}{p}\le B R^b\,\|g\|_{2}$ holds
for $B \gtrsim 1$ and $2\le p, q\le \infty$, and $b  \ge  \frac
np+\frac1q-1$. Then $\|\mathcal R_\Omega\, g\rnmix{q}{p}\le CB
R^b\,\|g\|_{2}$.
\end{lem}

We now return to the proof of Proposition \ref{linear}.

\subsection{Proof of \eqref{single} for  $(p,q)=(\infty, 2), $ $(\infty,4)$.}
We show that for $q=2,$ $4$,
\begin{equation}\label{qq}
\|\mathcal T_R^\nu h\rnmix{q}{\infty}\le C
\lambda_0^{-\frac14}(1+\nu)^{\frac25}
R^{\frac1q-\frac{2n-1}4}\|h\|_2.
\end{equation}
The estimate for  $q=2$ follows from the case $q=4$. In fact, by
Lemma \ref{local t} it is sufficient to show \eqref{qq} for $q=2$
when the time integration is taken over an  interval $I$ of side
length $\sim R$ but this follows from H\"older's inequality and the
estimate \eqref{qq} with $q=4$. Hence we need only to show that
\begin{equation}\label{l4}
\|\mathcal T_R^\nu h\rnmix{4}{\infty}\le C
\lambda_0^{-\frac14}(1+\nu)^{\frac25} R^{-\frac{n-1}2}\|h\|_2.
\end{equation}
For this we consider the following three cases, separately:
\begin{equation}
\label{cases} (1):  R\ll \nu, \ \  (2): \nu\lesssim R\lesssim
\nu^\frac85, \  \ (3):  \nu^\frac85\ll R.
\end{equation}

\noindent{\it Case $(1)$.} From \eqref{bessel1}  we have $\|\mathcal T_{R}^\nu h \rnmix{\infty}{\infty} \lesssim e^{-C\nu}R^{-\frac{n-2}2}\|h\|_2$. By Lemma \ref{local t} we get
\[\|\mathcal T_{R}^\nu h\rnmix{4}{\infty}\lesssim e^{-C\nu}R^{\frac14-\frac{n-2}2}\|h\|_2 \lesssim e^{-C\nu}\nu^{\frac34} R^{-\frac{n-1}2}\|h\|_2 \lesssim R^{-\frac{n-1}2}\|h\|_2,\]
which is acceptable.

\smallskip

\noindent{\it Case $(2)$}. By \eqref{lrinfty} we have $\|\mathcal T_{R}^\nu h \rnmix{\infty}{\infty}\lesssim R^{-\frac{n-1}2}\|h\|_2$. Then, by  Lemma \ref{local t}  and H\"older's inequality $\|\mathcal T_{R}^\nu h\rnmix{4}{\infty}\le CR^{\frac14-\frac{n-1}2}\|h\|_2$. So, if $R \sim \nu$, then
$\|\mathcal T_{R}^\nu h\rnmix{4}{\infty}\lesssim \nu^\frac14 R^{-\frac{n-1}2}\|h\|_2$. If $\nu \ll R \lesssim \nu^\frac85$, then
 $\|\mathcal T_{R}^\nu h\rnmix{4}{\infty}\lesssim \nu^\frac25 R^{-\frac{n-1}2}\|h\|_2$. So we get \eqref{l4}.

\smallskip

\noindent{\it Case $(3)$.}
For simplicity let us set
\begin{equation}\label{tom}
\m T_\Omega h(t,r)=r^{-\frac{n-2}2}\int e^{-it\varpi(\rho)} \Omega(r\rho)\beta(\rho)h(\rho) d\rho,
\end{equation}
and
\[\mathcal T_{\Omega, R} h(t,r)= \chi_R(r)\m T_\Omega h(t,r).\]
Since $r\sim R$,  using \eqref{bessel4}  we  need to consider
$\mathcal R_{\Omega}$ with
 \[\Omega(\rho)=\rho^{-\frac12} e^{\pm i\rho}, \ \nu^2 \rho^{-\frac32} e^{\pm i\rho}, \ \Psi_\nu(\r)=O(1/\rho)\]
and for \eqref{l4}  it is sufficient to show that
\[\|\chi_I(t) \mathcal T_{\Omega, R} h\rnmix{4}{\infty}\le C\lambda_0^{-\frac14}(1+\nu)^\frac25 R^{-\frac{n-1}2}\|h\|_2.\]
For $\Omega(\rho)=O(1/\rho)$, by Schwarz's inequality $|\mathcal
T_{\Omega, R} h(t,r)| \le CR^{-\frac n2}\|h\|_2$. So we get the
required bound from H\"older's inequality. Hence we only need to
consider the cases $\Omega(\rho)=\rho^{-\frac12} e^{\pm i\rho},$
$\nu^2 \rho^{-\frac32} e^{\pm i\rho}$. These two cases can be
handled similarly. In fact, since $\nu^\frac85\ll R$,
 we get the desired bound \eqref{l4}\footnote{The bound $\nu^\frac{8}{5}$ is actually decided by
the term $\nu^2 \rho^{-\frac32} e^{\pm i\rho}$. } if we show that
\[\|{\mathcal T} h\rnmix{4}{\infty}\le C\lambda_0^{-\frac14}\|h\|_2,\]
where
\[ {\mathcal T}h(r,t)=\w\chi_R(r)\int e^{i(-t\varpi(\rho)\pm r\rho)} \beta(\rho) h(\rho) d\rho.\]
By duality it is equivalent to
$\|{\mathcal T}^* H\|_{L^2}\le C\lambda_0^{-\frac14}\|H\rnmix{\frac43}{1},$
where $\mathcal T^*$ is the adjoint operator of $\mathcal T$. It again follows from
\[ \|\mathcal T\mathcal T^* H\rnmix{4}{\infty}\le C\lambda_0^{-\frac12} \|H\rnmix{\frac43}{1}.\]
Now note that
\begin{align*}
\mathcal T\mathcal T^* H(t,r)=\iint\mathcal K(t-s, r,r') [r'^{(n-1)}H(s,r')]dsdr',
\end{align*}
where
\[\mathcal K(t,r)=\w\chi_R(r)\w\chi_R(r') \int e^{i(-t\varpi(\rho)\pm (r-r')\rho)} \beta^2(\rho)d\rho.\]
Since $\lambda_0 \le |\varpi''(r)| \lesssim 1$, by van der Corput (see for instance page 334 of \cite{st}), it follows that
$
|\mathcal K(t,r)|\le
C\lambda_0^{-\frac12}|t|^{-1/2}.$
So  we get
\[\|\mathcal T\mathcal T^* H\rnmix{4}{\infty}\le C\lambda_0^{-\frac12}\bigg\|\int |t-s|^{-\frac12} \|H(\cdot, s)\|_{\mathfrak L_r^1} ds\bigg\|_{L^4_t}.\]
Then by Hardy-Littlewood-Sobolev inequality we get the desired
bound. This completes the proof of \eqref{single}, and hence
Proposition \ref{linear}. \qed

\begin{rk}\label{remark} From the  above proof ({\it Case }(3)) it is obvious
that if  $\Omega = \rho^{-\frac12} e^{\pm i\rho}$, $\Psi_\nu(\r)=O(1/\rho)$, then for $R\gtrsim 1$, $2\le p,q\le \infty$ and $2/q\ge 1/2-1/p$,
\begin{equation}\label{nunu1}
\|\mathcal  T_{\Omega, R} h\rnmix{q}{p}\lesssim
\lambda_0^{-\frac12(\frac12-\frac1p)}R^{\frac1q-\frac{2n-1}2(\frac12-\frac1p)}\|h\|_2,
\end{equation}
and if $\Omega = \nu^2 \rho^{-\frac32} e^{\pm i\rho}$,  for $2\le q\le 4$
\begin{equation}
\label{ome1}\|\mathcal  T_{\Omega, R} h\rnmix{q}{\infty}\lesssim
\lambda_0^{-\frac14} \nu^2 R^{-1}R^{\frac1q-\frac{2n-1}4}\|h\|_2.
\end{equation}
By  \eqref{single},  \eqref{bessel4} and \eqref{nunu1} for $\Omega = \rho^{-\frac12} e^{\pm i\rho}$, $\Psi_\nu(\r)=O(1/\rho)$, we also have for  $\Omega = \nu^2 \rho^{-\frac32} e^{\pm i\rho}$,
\begin{equation}\label{ome2}
\|\mathcal  T_{\Omega, R} h\rnmix{q}{2}\lesssim R^\frac1q\|h\|_2.
\end{equation}
Hence, by  interpolation between \eqref{ome1} and \eqref{ome2}  we see that if  $\Omega = \nu^2 \rho^{-\frac32} e^{\pm i\rho}$,
\begin{align}\label{nunu0}
\|\mathcal  T_{\Omega, R} h\rnmix{q}{p}\lesssim
\lambda_0^{-\frac12(\frac12-\frac1p)} (\nu^2
R^{-1})^{1-\frac2p}R^{\frac1q-\frac{2n-1}2(\frac12-\frac1p)}\|h\|_2
\end{align}
provided that  $2\le p,q\le \infty$ and $2/q\ge 1/2-1/p$. Hence,
when $R\ge 2\nu^2$,
 one gets  uniform bounds so that  if $\Omega = \nu^2 \rho^{-\frac32} e^{\pm i\rho}$,
$\rho^{-\frac12} e^{\pm i\rho}$, $\Psi_\nu(\r)=O(1/\rho)$,
\begin{equation}\label{nunu}
\|\mathcal  T_{\Omega, R} h\rnmix{q}{p}\lesssim
\lambda_0^{-\frac12(\frac12-\frac1p)}R^{\frac1q-\frac{2n-1}2(\frac12-\frac1p)}\|h\|_2,
\end{equation}
whenever $2\le p,q\le \infty$ and $2/q\ge 1/2-1/p$. \end{rk}

\subsection{An improvement on angular regularity} In what follows we
improve the bound in $\nu$  but  at the expense of losing a power of
$\lambda_0$ in the bound. This is why we need the extra condition on
$\omega$ in Theorem \ref{opt2}.

\begin{prop} \label{linear-2} Let $R\gtrsim 1$ and $\mathcal T_R^\nu$ be defined by \eqref{th-dfn}.  If
$2\le p,q\le \infty$, $\nu \ge 0$, and $1/q\ge 1/2-1/p$, then there
is a constant $C = C(n, p, q) > 0$, independent of $\lambda_0, \nu,$
$R$, such that
\begin{equation}\label{single-2}\|\mathcal T_R^\nu h\rn\le
C \lambda_0^{-(\frac{1}{2}-\frac1p)}(1+\nu)^{\frac12(\frac12 -
\frac1p)}
R^{\frac1q+\frac{2n-1}{2}(\frac1p-\frac12)}\|h\|_2.
\end{equation}
\end{prop}

\begin{proof}
\label{improve} From the proof Proposition \ref{linear} (see
\eqref{lrinfty} and \eqref{lr2}) and Remark \ref{remark} (see
\eqref{nunu}) we recall that the estimates \eqref{single-2} for
$(p,q)=(2,2),$ $(2,\infty)$ are already obtained. Hence, for the
proof of Proposition \ref{linear-2} it is sufficient to show
\eqref{single-2} for $(p,q)=(\infty,2).$ By Lemma \ref{local t} this
follows from
\begin{equation}\label{single11}\|\chi_I(t) \mathcal T_R^\nu h\rnmix{2}{\infty}\le
C \lambda_0^{-\frac12}(1+\nu)^\frac14 R^{-\frac{2n-3}{4}}\|h\|_2.
\end{equation}
Here $I$ is an interval of length $\sim R$. For the case $R \ll \nu$
it is easy to check \eqref{single11} as before and the case $R \gg
\nu^2$ is already handled (see \eqref{nunu} in Remark \ref{remark}). Hence to show
\eqref{single11}  we may assume
\[\nu\ll R\lesssim \nu^2.\]

To treat this case we use \eqref{bessel2nu1}. The contribution from
$h_\nu$ in \eqref{bessel2nu1} is $O(R^{-\frac{n-1}2}\|h\|_2)$. So,
it is acceptable.  Hence it is enough to show that
\begin{equation*}
\|\mathcal T_{\pm} h\rnmix{2}{\infty}\le C
\lambda_0^{-\frac12}\nu^{\frac14} R^{\frac{1}{4}}\|h\|_2,
\end{equation*}
where
\[ \mathcal T_{\pm} g(t,r)= \chi_I(t)\widetilde \chi_R(r)\int e^{-it\widetilde{\omega}(\rho)\pm i\theta(r\r)} \beta_\nu(\rho,r)g(\rho) d\rho\]
and $\beta_\nu(\rho,r)=\beta(\rho)\big(1-\frac{\nu^2}{\r^2
r^2}\big)^{-\frac14}$.
 We only show the estimate for
$ \mathcal T_+$. The other can be handled similarly.
  Following the previous argument
we need to show that
\[\|\mathcal T_+\mathcal T_+^* H\rnmix{2}{\infty}\le C\lambda_0^{-1}(1+\nu)^\frac12
R^{\frac{1}{2}}\|h\rnmix{2}{1}.\] Since
\[\mathcal T_+\mathcal T_+^* H=\iint \chi_I(t) \chi_I(s)K(t-s,r,r') [ r'^{n-1} H(s,r')] dr'ds,\]
and
\[K(t,r,r')=\int e^{-it\widetilde{\omega}(\rho)+ i(\theta(r\r)-\theta(r'\r))} \beta_\nu(\rho,r)\beta_\nu(\rho,r') d\rho.\]

Now let us observe that  for $\nu \ll r,$  $\rho \sim 1$
\[\Big|\frac{d^2}{d\rho^2} \theta(r\rho) \Big|\lesssim \frac{\nu^2}r \lesssim \nu.\]
So, if $|t|\ge C\lambda_0^{-1}\nu$ for some large $C$,
\[\Big|\frac{d^2}{d\rho^2}\Big(-t\widetilde{\omega}(\rho)+ \theta(r\r)-\theta(r'\r)\Big)\Big|\ge C\lambda_0|t|.\]
Hence, from van der Corput lemma we get $ |K(t,r,r')|\le
C\lambda_0^{-\frac12}|t|^{-\frac12}$ if $|t|\ge C\lambda_0^{-1}\nu$.
Hence using trivial bounds $ |K(t,r,r')|=O(1)$ for $|t|\le
C\lambda_0^{-1}\nu$ we see that
\[\int  \chi_I(t) \chi_I(s)\sup_{r,r'} |K(s-t,r,r')| dt, \int  \chi_I(t) \chi_I(s)\sup_{r,r'} |K(s-t,r,r')| ds \]
are bounded by
\[ C\int_0^{\lambda_0^{-1}\nu}  dt+ C\lambda_0^{-\frac12}\int_0^R t^{-\frac12} dt\le C\lambda_0^{-1}\nu+ C\lambda_0^{-\frac12}R^\frac12\le C \lambda_0^{-1} \nu^{\frac12}R^\frac12\]
because $\nu \ll R$. Then by Schur's test we get the desired bound.
\end{proof}

\begin{rk}[The wave equation]\label{wavecase}
For the wave equation $\omega(\rho)=\pm\rho$, the estimates are much easier to show. Let us consider the operator
\[
\mathcal W_R^\nu h(t, r) = \chi_R(r)r^{-\frac{n-2}2}\int
e^{-it\varpi(\rho)} J_\nu(r\rho)\beta(\rho)h(\rho) d\rho.
\]
Then we have for $2\le p,q\le \infty$
\begin{equation}\label{wave}
\|\mathcal W_R^\nu h\rn \le CR^{\frac1q+\frac{n-1}{p}-\frac{n-1}2} \|h\|_2.
\end{equation}
We only need to show the estimates for $(p,q)=(2,\infty),$ $(2,2)$,
$(\infty,\infty)$, $(2,\infty)$. In fact, the case
$(p,q)=(2,\infty)$ is a consequence of Plancherel's theorem. So, we
can apply H\"older's inequality and Lemma \ref{local t} to the
estimate \eqref{wave}  with  $(p,q)=(2,\infty)$ to get \eqref{wave}
for $(p,q)=(2,2)$.  When $(p,q)=(\infty,\infty)$, the desired
estimate can be obtained by Schwarz's inequality and
\eqref{L2-bessel} (cf. \eqref{lrinfty}). So similarly the case
$(p,q)=(\infty,2)$  also follows by H\"older's inequality and Lemma
\ref{local t}.
\end{rk}

\section{Proofs of Theorem \ref{opt}, \ref{opt2}}

In this section we prove Theorem \ref{opt} and \ref{opt2} by making
 use of the estimates in the previous section. The estimates other
than those on the sharp line
$(\frac{1}{q}=\frac{2n-1}{2}(\frac12-\frac1p))$ are relatively easy
to show once we have  Proposition \ref{linear}. However, to get the
endpoint estimates on the sharp line we show improved estimates
(Lemma \ref{interaction}) when the difference of spatial scales is
large and combine them with bilinear interpolation argument which
was used by Keel and Tao \cite{kt} to show the endpoint Strichartz
estimate.

We start with proving the necessity of the condition  $\frac1q \le
\frac{2n-1}{2}(\frac12 - \frac1p)$ for \eqref{homo}.

\subsection{The  necessary condition \eqref{nec}}\label{neccon}   Let $\varphi$ be a radially symmetric function  such that $\widehat \varphi$ is supported in $\{|\xi|\sim 1\}$.   Since $\widehat \varphi$ is also radial, we can write the solution $u$ to \eqref{linear eqn}  as
 \[u(t, x)= C|x|^{-\frac{n-2}2}\int e^{-it\omega(\rho)} \rho^\frac n2 J_{\frac{n-2}2}(|x|\rho)\widehat \varphi(\rho)\,d\rho.\]
Fix  $\rho_0\in [1,2]$. For $R\gg 1$, let us choose $\varphi$ such
that
\[\widehat \varphi(\xi)= \rho^{-\frac {n-1}2}\phi(R^\frac12(\rho-\rho_0)), \ \rho=|\xi|,\]
where $\phi\in C_c^\infty(-1,1)$. By the asymptotic of Bessel
function \eqref{upper} we have
\begin{align} \label{assmp}
u(t, x)&= C|x|^{-\frac{n-1}2}\int e^{i(-t\omega(\r) + |x|\rho)}
\phi(R^\frac12(\rho-\rho_0)) d\rho
\\
& \  \  \  +C|x|^{-\frac{n-1}2}\int e^{i(-t\omega(\r) - |x|\rho)}
\phi(R^\frac12(\rho-\rho_0))d\rho+
O(R^{-\frac12}|x|^{-\frac{n+1}{2}})\nonumber
\end{align}
provided $|x|\sim |t| \gg 1$.  Now observe that if $ |t|\le R$ and
$||x|-\omega(\rho_0) t|\lesssim  R^\frac12$,
\begin{align*}
-t\omega(\r) + |x|\rho
&=|x|\rho_0-t\omega(\r_0)+(|x|-\omega(\rho_0) t)(\rho-\rho_0)+O(t(\rho-\rho_0)^2)\\
&=|x|\rho_0-t\omega(\r_0)+O(1)
\end{align*}
since $|\rho-\rho_0|\le R^{-\frac12}$.  By changing the variables
$\rho\to \rho+\rho_0$ the second integral  equals
\[C|x|^{-\frac{n-1}2}e^{-i|x|\rho_0}  \int e^{i(-t\omega(R^{-\frac12}\r+\r_0) - R^{-\frac12}|x|\rho)}  \phi(\r)d\rho.\]
Since $|\frac{d}{d\r}(-t\omega(R^{-\frac12}\r+\r_0) -
R^{-\frac12}|x|\rho)|\ge CR^\frac12$  if $ t, |x|\sim R$,
 by integration by parts  we see that the second integral in \eqref{assmp} is $O(R^{-M})$ for any $M$ if $t, |x|\sim
 R$.
Hence, for $ t, |x|\sim R$ and $||x|-\omega(\rho_0)
t|\lesssim  R^\frac12$
\[|u(t, x)|\gtrsim R^{-\frac n2}.\]
Therefore it follows that
\[\| u\srmix{q}{p}\gtrsim R^{-\frac n2 +\frac1q +\frac{2n-1}{2p}}.\]
On the other hand $\|\varphi\|_{L_x^2}\sim R^{-\frac14}$.  Since
$\varphi$ is a radial function, $\|\varphi\|_{L_x^2} =
\|\varphi\hmix 2$. So the estimate \eqref{homo} implies that
$R^{-\frac n2 +\frac1q +\frac{2n-1}{2p}}\lesssim
 R^{-\frac14}$. Letting $R\to \infty$, we get the condition $\eqref{nec}. $

\subsection{Frequency localization}  By Littlewood-Paley theory,  scaling and orthogonality  the estimate \eqref{homo} can be obtained from the estimates for the simpler
operator $\mathcal T^\nu  $ which is defined by
\[\mathcal T^\nu  h(t,r) = r^{-\frac{n-2}2}\int e^{-it\varpi(\rho)} J_\nu(r\rho)\rho^\frac n2 \beta(\rho) h(\rho) d\rho.\]

\begin{lem}\label{flocal} Let $2 \le p<\infty$, $2\le q \le \infty$, $\gamma\ge 0$, and $\varpi \in C^4(1/2,2)$ which satisfies \eqref{cond ome 1st}. Suppose that for
$\nu = \nu(k) = \frac{n-2+2k}{2}$, $k\ge 0$,
\begin{equation}\label{reduced}
\|\mathcal T^\nu h\rn\le
C(1+\nu)^\gamma\lambda_0^{\frac1{2p}-\frac14}\|h\|_{2}.
\end{equation}
Then the solution $u$ to \eqref{linear eqn} satisfies \eqref{homo}
with $s_1 $, $s_2$ and $s $ satisfying \eqref{scon} provided that
$\alpha = \gamma+(n-1)(\frac12-\frac1p)$.
\end{lem}

\begin{proof}
Let $N>0$ denote dyadic numbers and let $\beta\in C_c^\infty(1/2,2)$
be such that $\sum_{N} \beta(|\xi|/N)=1$, $|\xi| \neq 0$. Then we
define  $P_N$ to be the projection operator given by
\[\widehat {P_N f} (\xi)=\beta\big(\frac{|\xi|}N\big)\widehat f(\xi).\]
Since $2\le p<\infty$ and $q\ge 2$, by Littlewood-Paley theory,
Minkowski's inequality and Sobolev embedding on the unit sphere
$S^{n-1}$  it follows that
\begin{align*}
\| \su \varphi \nmix{q}{p}
&\sim \Big\|\Big\|(\sum_N |P_N \su \varphi|^2)^\frac12 \Big\|_{L_x^p}\Big\|_{L_t^q} \\
&\le  \Big(\sum_N\| \su P_N\varphi \nmix{q}{p}^2\Big)^\frac12\\
&\le  \Big(\sum_{N}\| \su P_ND_\s^{(\alpha-\gamma)}
\varphi\srmix{q}{p}^2\Big)^\frac12.
\end{align*}
Note that $\alpha-\gamma=(n-1)(\frac12-\frac1p)$. Then, by
orthogonality it is sufficient for  \eqref{homo} to show that
\[
\| \su P_N\varphi\srmix{q}{p}\le C\|\mathcal D_\omega^{s_1,s_2}
P_N\varphi\|_{\dot{H}_r^sH_\s^{\gamma}}
\]
with $C,$ independent of $N$. By the property $(ii)$ of $\omega$ it reduces to
\[
\| \su P_N\varphi\srmix{q}{p}\le W_NN^{n(\frac 12 - \frac 1p)-
\frac2q }\|\varphi\|_{{\mathfrak L^2_r} H_\s^{\gamma}},
\]
where $W_N=C(\omega'(N){N}^{-1})^{(\frac14 - \frac1{2p})-\frac1q}
|\omega''(N)|^{\frac1{2p}-\frac14}.$ By rescaling $ \xi\to N\xi,$ $
x\to N^{-1} x,$ $t\to (N\omega'(N))^{-1} t, $ and \eqref{scon} this
is equivalent with
\begin{equation}\label{fflocal}
\| e^{-it \varpi(|\nabla|)} P_1\varphi\srmix{q}{p}\le
C\lambda_0^{\frac1{2p}-\frac14}\|\varphi\|_{{\mathfrak L^2_r}H_\s^\al},
\end{equation}
where
\begin{equation}\label{www} \varpi(\rho)=\frac{\omega(N\rho)}{N\omega'(N)}, \ \ \ \lambda_0=\Big|\frac{N\omega''(N)}{\omega'(N)}\Big|.\end{equation}

Since $\|\varphi\|_{{\mathfrak L^2_r}H_\s^\al}=\|\widehat\varphi\|_{{\mathfrak L^2_r}H_\s^\al}$ by
Plancherel's theorem and orthogonality of spherical harmonics, we
are reduced to showing that
\begin{align*}
\| T f\srmix{q}{p}\le
C\lambda_0^{\frac1{2p}-\frac14}\|f\|_{{\mathfrak L^2_r}H_\s^{\gamma}},
\end{align*}
for $f$ supported in $\{\frac12 \le |\xi| \le 2\}$, where
\[
 T f(t, x) = \int e^{i(x\cdot \xi - t \varpi(|\xi|))}\beta(|\xi|) f(\xi)\,d\xi.
\]

 We now expand $f$  by the orthonormal basis $\{Y_k^l\}, k \ge 0, 1 \le l \le d(k)$ of spherical harmonics
 (here $d(k)$ is the dimension
of spherical harmonics of order $k$) such that
 $$f(\xi) = f(\rho \s) = \sum_{k \ge 0}\sum_{ 1 \le l \le d(k)}a_k^l(\rho) Y_k^l(\s).$$
 We use the identity   $\widehat{Y_k^l}(\rho\s) =
c_{n,k}\rho^{-\frac{n-2}2}J_\nu(\rho)
 Y_k^l(\s)\footnote{$c_{n,k}=(2\pi)^\frac n2 i^{-k}$},$ $ \nu =
\nu(k) = \frac{n-2+2k}{2}$ (see \cite{StWe}) to get
\begin{align}
T f(t, x) = \sum_{k,l}c_{n,k}\mathcal T^\nu (a_k^l)(t, r)\;
Y_k^l(x/|x|),\quad r = |x|,\label{express}
\end{align}
where $|c_{n,k}| =(2\pi)^\frac n2 $, $k \ge 0$ for some positive
constant $C$ which is not depending on $k$. By orthogonality among
$\{Y_k^l\}$ and Minkowski's inquality
\[\| T f\srmix{q}{p}\le  C\Big( \sum_{k,l}\|\mathcal T^{\nu(k)} (a_k^l)\srmix{q}{p}^2\Big)^\frac12. \]
Since $\omega$ satifies the conditions $(i)-(iii)$, it is easy to
check that $\varpi$, $\lambda_0$ in \eqref{www} verifies the condition \eqref{cond
ome 1st}. Hence by the estimate \eqref{reduced} and the identity $
\|f\|_{L^2_rH_\s^\al} = \Big\|\Big(\sum_{k,l}(1 +
k(k+n-2))^\al|a_k^l|^2\Big)^\frac12\Big\|_{ L_{\rho\sim 1}^{2}} $
which follows from the fact that $-\De_\s Y_k^l = k(k+n-2)Y_k^l$, we
get
\[\| T f\srmix{q}{p}\le  C\lambda_0^{\frac1{2p}-\frac14}\Big( \sum_{k,l}(1+\nu(k))^{2\gamma}\|a_k^l\|_{ L_{\rho\sim 1}^{2}}^2\Big)^\frac12\le C\lambda_0^{\frac1{2p}-\frac14}\|f\|_{L^2_rH_\s^\gamma} . \]
This completes the proof.\end{proof}

\subsection{Proof of Theorem \ref{opt}} \label{thmproof} From the result in \cite{choozxia},
we already have estimates \eqref{homo} for $\frac
n2(\frac12-\frac1r)=\frac1q$ with $\alpha=0$. So, by interpolation
it is enough to consider estimates  near or on the sharp line
($\frac1q =\frac{2n-1}{2}(\frac12 - \frac1p)).$ Hence by Lemma
\ref{flocal} we only need to show \eqref{reduced} with
$\gamma=\frac45(\frac12-\frac1p)+\epsilon$ for any  $\epsilon>0$ if
$\frac n2(\frac12-\frac1r)<\frac1q \le  \frac{2n-1}{2}(\frac12-
\frac1p) $ and $q\neq 2$. In fact, note  that
$(n-1+\frac45)(\frac12-\frac1p)\to \frac{5n-1}{5(2n-1)}  $ as
$(p,q)\to (\frac{4n-2}{2n-3}, 2)$. So, we interpolate  \eqref{homo}
with $(p,q)$ arbitrarily   close to $(\frac{4n-2}{2n-3}, 2)$ and
\eqref{homoo} to get the desired estimate.

The rest of this subsection is devoted to the proof of
\eqref{reduced}.

\subsubsection{Estimates away from the sharp line}  We firstly show \eqref{reduced} when $\frac n2(\frac12-\frac1r)<\frac1q < \frac{2n-1}{2}(\frac12- \frac1p)$.
 We break the operator in spatial space radially. Fix a
dyadic number $ R_0\ge 1$. We write
 \[ \mathcal T^\nu h=\chi_{\{r<
R_0\}}\mathcal T^\nu h+ \chi_{\{r\ge  R_0\}}\mathcal T^\nu h .\] The
first is easy to handle. In fact, we show that  for $2\le p,q\le \infty$,
\begin{equation}\label{lolo} \| \chi_{\{r<
R_0\}}\mathcal T^\nu h\rn \le C\|h\|_2.
\end{equation}
From \eqref{express}  we note that
\[ c_{n,k}\mathcal T^\nu h(t,|x|)\, Y_k^l(\frac{x}{|x|})=\int e^{i(x\cdot \xi - t \varpi(|\xi|))}  h(|\xi|)Y_k^l(\frac{\xi}{|\xi|})\beta(|\xi|) \,d\xi.\]
Then the estimate \eqref{lolo} for $(p,q)=(2,\infty)$ follows from
Plancherel's theorem. Also, by taking $L^2$ norm in angular
variables (on $S^{n-1}$) and Schwarz's inequality  we  get
$|\mathcal T^\nu h(t,r)|\le C\|h\|_2.$ Interpolation establishes
\eqref{lolo} for $2\le p\le \infty,$ $q=\infty$. Now,   by Lemma
\ref{local t} it is sufficient for \eqref{lolo} to show $\|\chi_{[0,
2R_0]}(t)\mathcal T^\nu h(t,r)\rn \le C\|h\|_2$ for $2\le p,q\le
\infty$.  It follows by H\"older's inequality. Hence we get the
desired estimate \eqref{lolo}.

Recalling \eqref{th-dfn}, we further break $\chi_{\{r\ge
R_0\}}\mathcal T^\nu h$ to get
\[\chi_{\{r\ge R_0\}}\mathcal T^\nu  h=\sum_{R:\, dyadic, \ R \ge R_0 }\,\mathcal T_R^\nu h. \]
After triangle inequality we apply Proposition \ref{linear}
(estimate \eqref{single}) and sum  the resulting estimates  to get
\begin{align*}
\|\chi_{\{r\ge R_0\}}\mathcal T^\nu  h\rn \le C
\lambda_0^{-\frac12(\frac{1}{2}-\frac1p)}(1+\nu)^{\frac45(\frac12 -
\frac1p)} \|h\|_2
\end{align*}
provided that $2\le p,q\le \infty$, $2/q\ge 1/2-1/p$, and $\frac1q <
\frac{2n-1}{2}(\frac12- \frac1p)$.  From this and \eqref{lolo}  we
get the estimate \eqref{reduced}. This proves the non-endpoint
result due to Guo and Wang \cite{guwa}.

\subsubsection{Estimate along the sharp line $\frac1q+\frac{2n-1}{2p}-\frac{2n-1}{4}= 0$}
\newcommand{\schrd}{e^{i(t-s)\omega(|\nabla|)}}
We show the estimate \eqref{reduced}  for $\frac1q =
\frac{2n-1}{2}(\frac12- \frac1p)$, $q\neq 2$.  It will be basically
done by making use of $TT^*$ argument but the argument here is more
involved . Since the estimate for $(p,q)=(2,\infty)$ is trivial, we
may assume $p\neq 2$.

By \eqref{lolo} it is sufficient to consider $\chi_{\{r\ge
R_0\}}\mathcal T^\nu  h$. We further break it  (up to the cases in
\eqref{cases}) to get
\[\chi_{\{r\ge R_0\}}\mathcal T^\nu  h=\Big (\sum_{R_0\le R < 5\nu^\frac85}+\sum_{R\ge 5 \nu^\frac85} \Big)\,\mathcal T_R^\nu h. \]
The first sum is easy to handle. From Proposition \ref{linear}   we
have \eqref{single}.
 So, by direct summation we see that for $\frac1q+\frac{2n-1}{2p}-\frac{2n-1}{4}= 0$
\[ \|\sum_{R_0\le R< 5 \nu^\frac85} \mathcal T_R^\nu h\rn\le C(\log\nu)\nu^{\frac45(\frac12-\frac1p)}\|h\|_2.\]

To obtain the desired estimate  for $\sum_{R\ge  5 \nu^\frac85}
\mathcal T_R^\nu h$, it is sufficient to show that
\[ \|\chi_{\{r\ge 5\nu^\frac85\}} \m T_{\Omega} h\rn \le C(1+\nu)^{\frac45(\frac12 -
\frac1p)}\lambda_0^{\frac1{2p}-\frac14}\|h\|_2 \] with
$\Omega=J_\nu$. (See \eqref{tom}.)  For the proof of
\eqref{reduced}, using  \eqref{bessel4} in Lemma \ref{asymp bessel}
we need only  to show this with  \[\Omega(\rho)= \r^{-1/2} e^{\pm
i\r}, \, \nu^2\r^{-\frac32} e^{\pm i\r}, \, \Psi_\nu(\rho)=
O(\r^{-1}).\]

First we handle the case $\Omega=\Psi_\nu$. We break the operator
dyadically so that
\[\|\sum_{R\ge 5\nu^\frac85}\chi_R(r) \m T_{\Psi_\nu} h\rn\le \sum_{R\ge 5\nu^\frac85}\|\chi_R(r) \m T_{\Psi_\nu} h\rn.\]
Since $\Psi_\nu(\rho)= O(\r^{-1})$,  $\|\m T_{\Psi_\nu}
h\rnmix\infty\infty \le C R^{-\frac n2}$. From Lemma \ref{hh} and
H\"older's inequality,  it follows that  $\|\chi_R(r) \m
T_{\Psi_\nu} h\rnmix qp\le CR^{-\frac{n}{2}+\frac1q+\frac
np}\|h\|_2$ for $p,q\ge 2$. Since $\frac1q = \frac{2n-1}{2}(\frac12-
\frac1p)$, $p\neq 2$, we get for some $\epsilon>0$
\[\|\sum_{R\ge 5\nu^\frac85}\chi_R(r) \m T_{\Psi_\nu} h\rn\le C\sum_{R\ge 5\nu^\frac85}R^{-\epsilon}\|h\|_2 \le C\|h\|_2.\]
 When $\Omega(\rho)=\nu^2\r^{-\frac32} e^{\pm i\r}$, using \eqref{nunu0} we obtain the desired bound by direct summation.  Indeed, if  $\Omega(\rho)=\nu^2\r^{-\frac32} e^{\pm i\r}$, by \eqref{nunu0}
\begin{align*}
\|\sum_{R\ge 5\nu^\frac85}\chi_R(r) \m T_{\Omega} h\rn
&\le
C \lambda_0^{-\frac12(\frac12-\frac1p)} \sum_{R\ge 5\nu^\frac85} (\nu^2
R^{-1})^{1-\frac2p}R^{\frac1q-\frac{2n-1}2(\frac12-\frac1p)}\|h\|_2\\
&\le C(1+\nu)^{\frac45(\frac12 -
\frac1p)}\lambda_0^{\frac1{2p}-\frac14}\|h\|_2.
\end{align*}

We  now handle the case $\Omega(\rho)= \r^{-1/2} e^{\pm i\r}$ which
is the main term.   By discarding some irrelevant factors it is
sufficient to show that  for $\frac1q = \frac{2n-1}{2}(\frac12-
\frac1p)$, $2<q<\infty$,
\[ \label{eq1}\|\sum_{R\ge 5\nu^\frac85} S_Rh\rn\le C
\lambda_0^{\frac1{2p}-\frac14} \|h\|_2 ,\]
 where
\[
S_Rh(t,r)=R^{-\frac {n-1}2} \w\chi_R(r)\int e^{i(-t\varpi(\rho)\pm\,
r\r)} \beta(\rho)h(\rho) d\rho
\]
and $\varpi$ satisfies \eqref{cond ome 1st}.
By duality it is equivalent with
\begin{equation}\label{l22}
\|\sum_{R\ge 5\nu^\frac85} S_R^{\,\,*} H\|_2\le C
\lambda_0^{\frac1{2p}-\frac14} \|H\pprn
\end{equation}
for $\frac1q = \frac{2n-1}{2}(\frac12- \frac1p)$, $2<q<\infty$. Here
$ S_R^{\,\,*}$ is the adjoint of $ S_R$. Now the proof of Theorem
\ref{opt} completes if we show \eqref{l22}.

\subsubsection{An improved estimate for $S_{R}S_{R'}^{\,\,*} H$} From \eqref{nunu1} in Remark \ref{remark}  we have
\[\|S_R h\rnmix{q}{p}\le
C\lambda_0^{\frac1{2p}-\frac14} R^{\frac1q+\frac{2n-1}{2}(\frac1p-\frac12)}\|h\|_2\]
provided that $2\le p,q\le \infty$, $\nu \ge 0$, and $2/q\ge
1/2-1/p$. By duality  we have for $2\le
p,q\le \infty$ and $2/q\ge 1/2-1/p$
\begin{equation}
\label{dual} \|S_{R}^{\,\,*} H\|_2 \lesssim
\lambda_0^{-\frac12(\frac12-\frac1p)}R^{\frac1q-\frac{2n-1}2(\frac12-\frac1p)}
\|H\pprn.
\end{equation}
Hence it follows that
\begin{equation}
\label{inpest} \| S_{R}S_{R'}^{\,\,*} H\rn \lesssim
C\lambda_0^{(\frac1{2p}-\frac1{2\widetilde p'})}
R^{\frac1q-\frac{2n-1}2(\frac12-\frac1p)}(R')^{\frac1{\widetilde
q}-\frac{2n-1}2(\frac12-\frac1{\widetilde p})} \|H\prn
\end{equation}
provided that $2\le p,q\le \infty$ and $2/q\ge 1/2-1/p$ and $2\le
\widetilde p,\widetilde q\le \infty$ and $2/\widetilde q\ge
1/2-1/\widetilde p$. However to get the estimates at the critical
line \eqref{inpest} is still not enough. To get over it, we make an
observation which is stated in the following lemma.

\begin{lem}\label{interaction} Let us  denote $\max(R, R')$ by $R^*$ and $\min(R,R')$ by $R_*$.
If\, $2\le q, \widetilde q\le \infty$,
\begin{align}
\|S_R S_{R'}^{\,\,*} H\nqp{q}{\infty} &\le C\lambda_0^{-\frac12}
R^{\frac1q-\frac{2n-1}4} (R')^{\frac1{\widetilde q}-\frac{2n-1}4}
\Big(\frac{R_*}{R^*}\Big)^{\min(\frac{1}{4}-\frac1q,
\frac14-\frac1{\widetilde q})}\|H\nqp{\wtq}{1}\label{est1}.
\end{align}
\end{lem}

\begin{proof}
By \eqref{inpest} we may assume that $R^*\ge 8R_*$. Note that
\begin{align*}
S_R S_{R'}^{\,\,*} H(t,r)=\iint K_{R,R'}(t-s,r,r')
[(r')^{n-1}H(s,r')]dsdr',
\end{align*}
where
\[K_{R,R'}(t,r,r')=(RR')^{-\frac{n-1}{2}}\widetilde\chi_R(r)\widetilde\chi_{R'}(r') \int e^{i(-t\varpi(\rho)\pm\, (r-r')\r)}\beta^2(\rho)d\rho.
\]
We first break the kernel $K_{R,R'}$ so that
\[K_{R,R'}(t,r,r')= K_1(t,r,r')+ K_2(t,r,r'),\]
where \[K_1(t,r,r')=\chi_{\{R^*/8\le |t|\le 8R^*\}}K_{R,R'}(t,r,r')
.\] Since $|\varpi'(\rho)|\lesssim 1$,
$|\frac{d}{d\rho}(-t\varpi(\rho)\pm\, (r-r')\r)|\ge  C\max(|t|, R^*)
$ if $|t|\le R^*/8$ or $|t|\ge 8R^*$. Hence by  integration
by parts (three times) we see that
 \[|K_2(t,r,r')|\le (R^*)^{-a}(1+t)^{-(3-a)}\] for any
$0 \le a \le 3$. So,
 the contribution from $K_2$ is negligible. It is now sufficient for \eqref{est1} to consider the operator
 \[
S_{R,R'} H(t,r)=\iint K_1(t-s,r,r') [(r')^{n-1}H(s,r')]dsdr'
\]
instead of $S_R S_{R'}^{\,\,*}$. Since $|\varpi''|\gtrsim \lambda_0$
and $K_1(\cdot, r,r')$ is supported in $[R^*/8, 8R^*]$, by the van
der Corput lemma it follows that
\[|K_1(t,r,r')|\le C(R^*\lambda_0)^{-\frac12}.\]
By the standard argument, obviously we may assume that the temporal
supports of  $S_{R,R'} H,$ $H$ are contained in an interval of
length $\sim R^*$. By H\"older's inequality and the above kernel
estimate we have for $1\le q\le 2\le \wtq\le \infty$
\begin{align*}
\|S_{R,R'} H\nqp{q}{\infty} &\le C(R^*)^{\frac1q+1-\frac1\wtq}
(RR')^{-\frac{2n-2}{4}}
(R^*\lambda_0)^{-\frac12}\|H\nqp{\wtq}{1}\nonumber
\\
&\le C\lambda_0^{-\frac12}R^{\frac1q-\frac{2n-1}4}
(R')^{\frac1{\widetilde q}-\frac{2n-1}4}
\Big(\frac{R_*}{R^*}\Big)^{\min(\frac{1}{4}-\frac1q,
\frac14-\frac1{\widetilde q})}\|H\nqp{\wtq}{1}\label{est1}.
\end{align*}
Hence we get the desired estimate \eqref{est1}.
\end{proof}

Now we interpolate \eqref{inpest}  and \eqref{est1} to get an
improvement on the estimate \eqref{inpest} when $R\not\sim R'$. In
particular, taking $p=\widetilde p=2$ in \eqref{inpest},  we have
\[
\| S_{R}S_{R'}^{\,\,*} H\nqp{q}{2} \lesssim C
R^{\frac1q}R'^{\frac1{\widetilde q}} \|H\nqp{\wtq}{2}
 \]
 provided that $2\le q,\widetilde q\le \infty$. Then we interpolate it with \eqref{est1}
 to get
\begin{align}\label{wtgain}
\|S_R S_{R'}^{\,\,*} H\rn \le &C\lambda_0^{\frac1{2p}-\frac1{2p'}}\min\big(\frac{R}{R'},\frac{R'}{R}\big)^{\epsilon}
\\ &\times
R^{\frac1q-\frac{2n-1}2(\frac12-\frac1p)}(R')^{\frac1{\widetilde
q}-\frac{2n-1}2(\frac12-\frac1{p})} \|H\rnmix{\wtq}{p'}\nonumber
\end{align}
for some $\epsilon=\epsilon(p,q,\widetilde q)>0$ provided that
$2<p\le \infty$ and $0\le \frac1q,\frac1{\widetilde
q}<\frac14+\frac1{2p}$. Clearly  we may assume that $\epsilon$
continuously depends on $\frac1p,\frac1q,\frac1{\widetilde q}$. So,
if $\Delta$ is a compact subset of
 $\{(\frac1p,\frac1q,\frac1{\widetilde q}): \frac12>\frac 1p\ge 0, \ 0\le \frac1q,\frac1{\widetilde q}<\frac14+\frac1{2p} \}$,
 there is a uniform lower bound $\epsilon_0=\epsilon_0(\Delta)$
such that $\epsilon(p, q,\widetilde q)\ge \epsilon_0>0$ if $(1/p,
1/q, 1/\widetilde q)\in \Delta$.

\subsubsection{Proof of \eqref{l22} for $2<p<\frac{2n}{n-1}$}
We firstly show \eqref{l22} for $p<\frac{2n}{n-1}$. The remaining
case will be handled differently. For \eqref{l22} it suffices to
show that
\begin{equation}
\label{eq2}\|\sum_{R, R'\ge 5\nu^\frac85} S_R S_{R'}^{\,\,*} H\rn\le
C\lambda_0^{(\frac1{2p}-\frac1{2p'})} \|H\pprn\end{equation}
provided that $\frac1q = \frac{2n-1}{2}(\frac12- \frac1p)$, $q\neq
2$, $p\neq 2$. Fix $p,q$ such that $\frac1q =
\frac{2n-1}{2}(\frac12- \frac1p)$, $2<p<\frac{2n}{n-1}$.  We write
 \[\sum_{R,R'\ge 5\nu^\frac85} S_R S_{R'}^{\,\,*} H=\sum_{k=-\infty}^\infty \Big(\sum_{R,R'\ge 5\nu^\frac85; \frac{R}{R'}=2^k} S_R S_{R'}^{\,\,*} H\Big).\]
 By  \eqref{wtgain} each of summand in the inner summation satisfies
  \begin{equation}\label{prev}
  \|S_R S_{R'}^{\,\,*} H\nqp{s}{p}\le C\lambda_0^{\frac1{2p}-\frac1{2p'}}
  (R'2^{k/2})^{(\frac1s+\frac1{\w s}-\frac{2n-1}{2}(1-\frac2p))} 2^{-\epsilon |k|}  \|H\nqp{\w s'}{p'}
  \end{equation}
for some $\epsilon>0$ if $0\le \frac1s,\frac1{\widetilde
s}<\frac14+\frac1{2p}$.

We now use a summation argument due to Bourgain \cite{Bour}. (Also
see \cite{CSWW}  for a generalization.) For reader's convenience we
state a version which we need here. (See \cite{LS} for a simple
proof.)
\begin{lem}\label{summation}
Let $\varepsilon_1,\varepsilon_2>0$.  Let $A, B$ be Banach spaces
and $1\le r_1, r_2, s_1, s_2<\infty$.  Suppose that
$\{T_j\}_{j=-\infty}^\infty$ be a collection of operators satisfying
that $\|T_j F\|_{L^{s_1}(B)}\le C M_12^{\varepsilon_1j}
\|F\|_{L^{r_1}(A)}$ and $\|T_j F\|_{L^{s_2}(B)}\le C
M_22^{-\varepsilon_2 j}  \|F\|_{L^{r_2}(A)}$. Then
\[\|\sum T_j F\|_{L^{s,\infty}(B)}\le CM_1^{\theta}M_2^{1-\theta} F\|_{L^{r,1}(A)},\]
where $\theta=\varepsilon_2/(\varepsilon_1+\varepsilon_2)$,
$1/r=\theta/{r_1}+(1-\theta)/{r_2}$ and
$1/s=\theta/{s_1}+(1-\theta)/{s_2}$.   Here $L^{r,a}$ denotes the
Lorentz space.
\end{lem}

Let us set
\[I_p=\Big \{(\frac1s, \frac1{\w s}): \frac1s+\frac1{ \w s}= \frac{2n-1}{2}(1-\frac2p) ,
\ 0< \frac1s, \ \frac1{\widetilde s}<\frac14+\frac1{2p}\Big\}.
\]
The open line segment  $I_p$  is not empty as long as $
\frac{2n-1}{2}(1-\frac2p) < \frac12+\frac1{p} $ (equivalently
$p<\frac{2n}  {n-1}$). By applying Lemma \ref{summation} with
\eqref{prev}, we get for $2<p< \frac{2n}  {n-1}$
\begin{equation}\label{weak1}
\| \sum_{R,R'\ge 5\nu^\frac85;\frac{R}{R'}= 2^k} S_R S_{R'}^{\,\,*}
H \nqp{s,\infty}{p}\le C2^{-\epsilon |k|} \|H\nqp{\w s',1}{p'}
\end{equation}
provided  $(\frac1s,\frac1{\widetilde s})\in  I_p$. Note that the
exponent of $(R'2^{k/2})$ is equal to zero if $\frac1s+\frac1{ \w
s}= \frac{2n-1}{2}(1-\frac2p)$. Since $\w s'<2 < s$, by real
interpolation among the estimates \eqref{weak1} for
$(\frac1s,\frac1{\widetilde s})\in I_p$, they  can
be strengthened to strong type. 
Hence, if $2<p<\frac{2n}{n-1}$ and $(\frac1s,\frac1{\widetilde
s})\in I_p$, then  we have
\begin{equation}\label{strong}
\| \sum_{R,R'\ge 5\nu^\frac85;\frac{R}{R'}= 2^k} S_R S_{R'}^{\,\,*}
H \nqp{s}{p}\le C\lambda_0^{\frac1{2p}-\frac1{2p'}}2^{-\epsilon |k|}
\|H\nqp{\w s'}{p'}.
\end{equation}
So, for $2<p<\frac{2n}{n-1}$ and $(\frac1s,\frac1{\widetilde s})\in
I_p$ we get
\begin{align*}
 \Big\|\sum_{k=-\infty}^\infty\Big(\sum_{R,R'\ge 5\nu^\frac85; \frac{R}{R'}=2^k} &S_{R}S^{\,\, *}_{R'} H\Big)\Big{\nqp{s}{p}}\le
C\lambda_0^{\frac1{2p}-\frac1{2p'}}\|H\nqp{\w s'}{p'}.
\end{align*}
   In particular, if we take $s =\w s$ $ (=q)$, we get the desired
   estimate \eqref{eq2} for $p<\frac{2n}{n-1}$.

\subsubsection{Proof \eqref{l22} for $\frac{2n}{n-1}\le p<\frac{2(2n-1)}{2n-3}$ }
After squaring the left hand side of \eqref{l22}, we rearrange it so
that
 \[\sum_{R,R'\ge 5\nu^\frac85} \inp{S_{R}^{\,\,*} H}{ S_{R'}^{\,\,*} H}=\sum_{k=-\infty}^\infty \Big(\sum_{R,R'\ge 5\nu^\frac85; \frac{R}{R'}=2^k} \inp{S_{R}^{\,\,*} H}{ S_{R'}^{\,\,*} H}\Big).\]
Hence the desired estimate \eqref{l22} follows if we show that  for
$2<p<\frac{2(2n-1)}{2n-3}$
 \begin{equation}
 \label{decompkk}
| \sum_{R,R'\ge 5\nu^\frac85; \frac{R}{R'}=2^k} \inp{S_{R}^{\,\,*}
H}{ S_{R'}^{\,\,*} G}|\le C 2^{-\epsilon |k|}
\lambda_0^{\frac1p-\frac12}\|H\rnmix{q'}{p'}\|G\rnmix{q'}{p'}.
 \end{equation}
From \eqref{strong} we already established this inequality for
$2<p<\frac{2n}{n-1}$ and $\frac1q=\frac{2n-1}{2} (\frac12-\frac1p).$
 To get \eqref{decompkk} for $\frac{2n}{n-1}\le p<\frac{2(2n-1)}{2n-3}$ it is sufficient to show that
  \begin{equation}
 \label{decompk}
 \Big|\sum_{R,R'\ge 5\nu^\frac85; \frac{R}{R'}=2^k} \inp{S_{R}^{\,\,*} H}{ S_{R'}^{\,\,*} G}\Big|\le C \lambda_0^{\frac{2n-3}{4n-2}-\frac12} \|H\rnmix{2}{\frac{4n-2}{2n+1}}\|G\rnmix{2}{\frac{4n-2}{2n+1}}.
 \end{equation}
In fact,  interpolating this with \eqref{decompkk} for
$2<p<\frac{2n}{n-1}$ we get \eqref{decompkk} for
$2<p<\frac{4n-2}{2n-3}$.

To show \eqref{decompk} we adopt bilinear interpolation argument in
\cite{kt}, which was used to show the endpoint Strichartz estimate.
Let us denote by $\ell_r^s$ the pace of sequences $\{Z_R\}_{R:
dyadic}$ with norm
\[ \|\{Z_R\}\|_{\ell_r^s}=\begin{cases} &\Big(\sum_{R: dyadic}  |R^s Z_R|^r \Big)^\frac1r \text{ if } r\neq \infty, \\
 &\quad\sup_{R: dyadic } |R^s Z_R|  \quad\text{ if } r= \infty.
                                       \end{cases}    \]
                                                       We will use the fact (see  Theorem 5.6.2 in  \cite{BL}) that
 if $0<q_0, q_1\le \infty$ and $s_0\neq s_1$, then for $q\le \infty$,
 \begin{equation} \label{interseq}(\ell_{q_0}^{s_0}, \ell_{q_1}^{s_1})_{\theta, q}=\ell_{q}^{s},\end{equation}
 where $s=(1-\theta)s_0+\theta s_1$. And we also recall the following  fact on real interpolation which is due to Lions and
 Peetre  \cite{lp} (also see \cite{Tri}, section 1.18.4):  Let $A_0$, $A_1$ be Banach spaces. If $1\le p_0, p_1<\infty$, $0<\theta<1$, and $\frac1{p}=\frac{1-\theta}{p_0}+\frac\theta{p_1}$, then
 \begin{equation} \label{intelp}( L^{p_0}(A_0), L^{p_1}(A_1))_{\theta, p} =L^p((A_0,A_1)_{\theta,p}).\end{equation}
Here $(A_0,A_1)_{\theta, r}$ denotes the real interpolation space.

 We now consider the bilinear operator which is defined by
\[B_k(H,G)_{R'}=
\begin{cases}
\inp{S_{R}^{\,\,*} H}{ S_{R'}^{\,\,*} G} \, &\text{ if } R=2^kR',\, R,R'\ge 5\nu^\frac85, \\
\qquad \quad 0  & \text{\,\, otherwise. }
\end{cases}\]
By \eqref{inpest} and duality,  taking $q=\widetilde q=2$
particularly,  we have for $2\le p,\widetilde p \le \infty$
\begin{align}
\label{inpest2}  |\inp{S_{2^k R}^{\,\,*} H}{ S_{R}^{\,\,*}
G}|\lesssim C\lambda_0^{\frac1{2p}-\frac1{2\widetilde p'}}
2^{k(\frac{2n-1}2\frac1p-\frac{2n-3}4)} R^{\frac{2n-1}2(\frac
1p+\frac1{\widetilde p})-\frac{2n-3}2}\|H\rnmix{2}{p'}
\|G\rnmix{2}{\widetilde p'}.
\end{align}
Let us set
 \[\beta(p,\widetilde p)=\frac{2n-3}2-\frac{2n-1}2(\frac 1p+\frac1{\widetilde p}).\]
Then \eqref{inpest2} implies that for $2\le p, \w p \le \infty$
\begin{align}\label{bibi}
B_k: L^2_t\mathfrak L_r^{p'}\times L^2_t\mathfrak L_r^{\w p'} \to
\ell_\infty^{\beta(p,\w p)}
\end{align}
is bounded with bound $C\lambda_0^{\frac1{2p}-\frac1{2\widetilde p'}}2^{k(\frac{2n-1}2\frac1p-\frac{2n-3}4)} $.

Now we apply the following interpolation lemma. See \cite{BL}
(exercise 5(b) in section 3.13).

\begin{lem}\label{interbii} Let $A_0,$ $A_1$, $B_0,$ $B_1$, $C_0,$ and $C_1$ be Banach spaces, and
$T$ be a bilinear operator such that
\begin{equation*}\begin{aligned}
&\|T(f,g)\|_{C_0}\le M_{0,0}\|f\|_{A_0}\|g\|_{B_0},\\
&\|T(f,g)\|_{C_1}\le M_{i,1-i}\|f\|_{A_i}\|g\|_{B_{1-i}}, \, i=0,1.
\end{aligned}
\end{equation*}
Then, if $0<\theta_a, \theta_b<1$, $\theta=\theta_a+\theta_b$, $1\le
u, v,r\le \infty$, and $1\le 1/u+1/v$,
\[ T: (A_0, A_1)_{\theta_a, up}\times (B_0, B_1)_{\theta_b, vq}\to  (C_0,C_1)_{\theta, r}\]
is bounded with norm $\lesssim
M_{0,0}^{1-\theta_a-\theta_b}M_{1,0}^{\theta_a}M_{0,1}^{\theta_b}$.
\end{lem}

 We choose $p_0,
$ $p_1\in [2,\infty)$ such that $p_0<p_1$ and  $\beta(p_0,p_1)<
\beta(p_0,p_0)$. Obviously such choices are always possible. We now
consider $A_i= L^2_t\mathfrak L_r^{p_i'},$ $B_i=L^2_t\mathfrak
L_r^{p_i'},$ $i=0,1,$ and $C_0=\ell_\infty^{\beta(p_0,p_0)},$ $
C_1=\ell_\infty^{\beta(p_0,p_1)}.$ Then by \eqref{bibi} $T=B_k$
satisfies the assumptions of Lemma \ref{interbii} with
\begin{align*} M_{0,0}= C\lambda_0^{\frac1{2p_0}-\frac1{2
p'_0}}&2^{k(\frac{2n-1}2\frac1{p_0}-\frac{2n-3}4)},\,\,\, M_{1,0}=
C\lambda_0^{\frac1{2p_1}-\frac1{2
p'_0}}2^{k(\frac{2n-1}2\frac1{p_1}-\frac{2n-3}4)},\\
&M_{0,1}= C\lambda_0^{\frac1{2p_0}-\frac1{2
p'_1}}2^{k(\frac{2n-1}2\frac1{p_0}-\frac{2n-3}4)}.
\end{align*}
By setting $r=1,$ $u=v=2$, we apply Lemma \ref{interbii}. Then it
follows that if $0<\theta=\theta_a+\theta_b<1,$
\[B_k: ( L^2_t\mathfrak L_r^{p_0'}, L^2_t \mathfrak L_r^{p_1'} )_{\theta_a,2}
\times( L^2_t\mathfrak L_r^{p_0'}, L^2_t\mathfrak L_r^{p_1'}
)_{\theta_b,2} \to  (\ell_\infty^{\beta(p_0,p_0)},
\ell_\infty^{\beta(p_0,p_1)})_{\theta, 1}
\]
with bound $C\lambda_0^{\frac12(\frac{2-\theta_a-\theta_b}{p_0} +
\frac{\theta_a+\theta_b}{p_1}-1)}2^{k(\frac{2n-1}2(\frac{1-\theta_a}{p_0}+\frac{\theta_a}{p_1})-\frac{2n-3}4)}
$. Here $\mathfrak L_r^{p,r}$ is the Lorentz space defined with
measure $r^{n-1} dr$. By setting
\[\Big(\frac1p,\frac1{\widetilde p}\Big)=(1-\theta_a-\theta_b)\Big(\frac1{p_0},\frac1{p_0}\Big)
+\theta_a\Big(\frac1{p_1},\frac1{p_0}\Big)
+\theta_b\Big(\frac1{p_0},\frac1{p_1}\Big)\] and by \eqref{interseq}
and \eqref{intelp}, we now have for $p_1<p,\w p<p_0$ satisfying
$1/p+1/{\w p}<1/p_0+1/p_1$
\[B_k:  L^2_t\mathfrak L_r^{p',2}  \times L^2_t\mathfrak L_r^{\w p',2}  \to  \ell_1^{\beta(p, \w p)}\]
with bound $C\lambda_0^{\frac1{2p}-\frac1{2\widetilde
p'}}2^{k(\frac{2n-1}2\frac1p-\frac{2n-3}4)} $.  Considering all the
possible choices of $p_0,$ $p_1$,  we see that this is also valid
for all $(p,\w p)$ satisfying $2<p,\w p<\infty$. Hence, in
particular, for $2<p,\w p<\infty$ satisfying
$\frac{2n-3}2-\frac{2n-1}2(\frac 1p+\frac1{\widetilde p})=0 $ $(
\beta(p,\widetilde p)=0)$ we have
\[| \sum_{ R,R'\ge 5\nu^\frac85; \frac{R}{R'}=2^k}
\inp{S_{R}^{\,\,*} H}{ S_{R'}^{\,\,*} G}|\le C
\lambda_0^{\frac1{2p}-\frac1{2\widetilde
p'}}2^{k(\frac{2n-1}2\frac1p-\frac{2n-3}4)}
\|H\rnmix{2}{p',2}\|G\rnmix{2}{\w p',2}.
\]
Taking $p=\w p $ $ (=\frac{2(2n-1)}{2n-3}) $,  we get the desired
\eqref{decompk} since $L^{2}_{t}\mathfrak L^{ p'}_{r}\subset
L^{2}_{t}\mathfrak L^{ p',2}_{r} $. This completes the proof.

\subsection{Proof of Theorem \ref{opt2}}
Theorem \ref{opt2} can be proven similarly as Theorem \ref{opt}.
Once we have Lemma \ref{linear-2}, we can routinely follow the
arguments for the proof of Theorem \ref{opt}. The only difference
comes from the additional assumption \eqref{simsim} (see
\eqref{fflocal} and \eqref{www}) by which we have $\lambda_0\sim 1$
at \eqref{www}. Hence we do not have any loss in $\lambda_0$ when
applying Lemma \ref{linear-2}. Then the  remaining is almost
identical with the proof of Theorem \ref{opt}.  We omit the detail.

\subsection{Remark for the wave equation}\label{wavet}
In \cite{ster}, the estimate \eqref{angle} was proven for
$\omega(\rho)=\rho$, $s=n(\frac12-\frac1p)-\frac1q$, $\alpha
> \frac2q -(n-1)(\frac12-\frac{1}{p})$ when $n\ge 3$ provided that
\begin{equation}
\frac{n-1}2(\frac12-\frac1p)<\frac1q < (n-1)(\frac12-\frac1p).
\end{equation}
It was shown by Knapp's example that the estimate fails if $
\frac1q>(n-1)(\frac12-\frac1p)$. The example in \cite{ster} also
shows that $\alpha \ge  \frac2q -(n-1)(\frac12-\frac{1}{p})$ is
necessary for \eqref{angle}.

Let us set
\[\mathcal W^\nu  h(t,r) = r^{-\frac{n-2}2}\int e^{-it\rho} J_\nu(r\rho)\rho^\frac n2 \beta(\rho) h(\rho) d\rho.\]
Similarly as before, it is easy to see that $\| \chi_{\{r<
R_0\}}\mathcal W^\nu h\rn \le C\|h\|_2.$ (See \eqref{lolo}.) Hence,
by \eqref{wave} it follows that  for $2\le p,q \le \infty$ and
$\frac1q < (n-1)(\frac12-\frac1p)$
\begin{equation}\label{reducedw}
\|\mathcal W^\nu h\rn\le C\|h\|_{2}.
\end{equation}
Using Lemma \ref{summation}, it seems possible to get some weak type
estimates for  $\mathcal W^\nu$ along the sharp line
$\frac1q=(n-1)(\frac12-\frac1p)$, $q>2$ but the strong type endpoint
estimates are not possible by the method in this paper because
$\omega''=0$. By the argument for the proof of Lemma \ref{flocal},
the following is easy to show.

\begin{lem}\label{flocalw} Let $\omega(\rho)=\rho$ and  $2 \le p<\infty$, $2\le q \le \infty$, $\gamma\ge 0$. Suppose that \eqref{reducedw} holds for
$\nu = \nu(k) = \frac{n-2+2k}{2}$, $k\ge 0$. Then the solution $u$
to \eqref{linear eqn} satisfies \eqref{angle} with
$s=n(\frac12-\frac1r)-\frac1q$ provided that $\alpha
=(n-1)(\frac12-\frac1p)$.
\end{lem}

Hence we get \eqref{angle} for  $2 \le p<\infty$, $2\le q \le
\infty$  and $\frac1q<(n-1)(\frac12-\frac1p)$ provided $\sigma\ge
(n-1)(\frac12-\frac1p)$. Now note $ (n-1)(\frac12-\frac1p)=\frac2q
-(n-1)(\frac12-\frac{1}{p})$ if $ \frac1q=(n-1)(\frac12-\frac1p)$.
Hence interpolating these estimates with the usual Strichartz
estimates for the wave equation (along the sharp line
$\frac1q=\frac{n-1}2(\frac12-\frac1p)$) recovers the aforementioned
results in \cite{ster}.  This  also shows that if one can obtain
\eqref{reducedw} on the sharp line $\frac1q=(n-1)(\frac12-\frac1p)$,
then the optimal angular regularity ($\alpha = \frac2q
-(n-1)(\frac12-\frac{1}{p})$) for \eqref{angle} also follows.

\

\noindent{\bf  Acknowledgement.} We would like to thank Kenji
Nakanishi for a comment on the paper. Y. Cho  was supported by NRF
grant 2011-0005122. S. Lee was supported by NRF grant 2011-0001251
(Korea).

\end{document}